\documentclass[a4paper,10pt]{amsart}

\usepackage{amsmath,amssymb,amsthm,a4wide,graphicx}

\usepackage{caption}
\newtheorem{thm}{Theorem}[section]
\newtheorem*{lem*}{Lemma}
\newtheorem*{rmk*}{Remark}

\newtheorem{lem}[thm]{Lemma}

\newcommand{\vf}{\mathbf{f}} 
\newcommand{\q}{\mathbf{q}} 
\newcommand{\K}{\mathbf{K}} 
\newcommand{\PP}{\mathbf{P}} 
\newcommand{\bR}{\mathbb{R}} 
\newcommand{\R}{\mathbb{R}} 
\newcommand{\vU}{\mathbf{U}} 
\newcommand{\cM}{\mathcal{M}} 
\newcommand{\M}{\mathcal{M}} 
\newcommand{\cX}{\mathcal{X}} 
\newcommand{\vPi}{\boldsymbol{\Pi}} 
\newcommand{\vpi}{\boldsymbol{\pi}} 
\newcommand{\vSigma}{\boldsymbol{\Sigma}} 
\newcommand{\vLambda}{\boldsymbol{\Lambda}} 
\newcommand{\vPsi}{\boldsymbol{\Psi}} 
\newcommand{\vPhi}{\boldsymbol{\Phi}} 
\newcommand{\vA}{\mathbf{A}} 
\newcommand{\vV}{\mathbf{V}} 
\newcommand{\vdelta}{\boldsymbol{\delta}} 
\newcommand{\T}{\mathcal{T}} %
\newcommand{\defeq}{:=} 
\newcommand{\Span}{\text{\normalfont Span}} 
\newcommand{\bigo}{\mathcal{O}} 

\title{Time Coupled Diffusion Maps}

\author[]{Nicholas F. Marshall}
\address{Department of Mathematics, Yale University, New Haven, CT 06511, USA}
\email{nicholas.marshall@yale.edu}

\author[]{Matthew J. Hirn}
\address{Department of Computational Mathematics, Science \&
  Engineering and Department of Mathematics, Michigan State
  University, East Lansing, MI 48824, USA}
\email{mhirn@msu.edu}


\begin{document}

\maketitle

\begin{abstract}
We consider a collection of $n$ points in $\R^d$ measured at $m$ times, which
are encoded in an $n \times d \times m$ data tensor. Our objective is to define
a single embedding of the $n$ points into Euclidean space which summarizes the
geometry as described by the data tensor. In the case of a fixed data set,
diffusion maps and related graph Laplacian methods define such an embedding via
the eigenfunctions of a diffusion operator constructed on the data.  Given a
sequence of $m$ measurements of $n$ points, we introduce the notion of time
coupled diffusion maps which have natural geometric and probabilistic
interpretations. To frame our method in the context of manifold learning, we
model evolving data as samples from an underlying manifold with a time-dependent
metric, and we describe a connection of our method to the heat equation on such
a manifold.
\end{abstract}

\section{Introduction} \label{sec:introduction}
In many machine learning and signal processing tasks, the observable data is
high dimensional, but it lies on a low-dimensional intrinsic manifold. In recent
years, several manifold learning methods have emerged which attempt to recover
the intrinsic manifold underlying datasets. In particular, graph Laplacian
methods have become popular due to their practicality and theoretical guarantees
\cite{Belkin:2003, Lafon:2004, Belkin:2005, Coifman:2006,
singer:GraphToManifold2006, singer:vectorDiffMaps2011, wolf:linearProjDiff2011,
berry:localKernels2015}. 

Current graph Laplacian methods implicitly assume a static intrinsic manifold,
or equivalently, that the dynamics underlying the data generation process are
stationary. For many applications, this stationary assumption is justified, as
datasets often consist of a single snapshot of a system, or are recorded over
small time windows. However, in the case where data is accumulated over longer
periods of time, accounting for changing dynamics may be advantageous.
Furthermore, if a system is particularly noisy, combining a large number of
snapshots over time may help recover structure hidden in noise. These
observations raise the following question: how can graph Laplacian methods be
extended to account for changing dynamics while maintaining theoretical
guarantees? 

In this paper, we propose modeling data with changing dynamics by assuming there
exists an underlying intrinsic manifold with a time-dependent metric.  We will
describe the proposed method using the diffusion maps framework: a popular graph
Laplacian framework which is robust to non-uniform sampling \cite{Coifman:2006}.
We remark that diffusion maps are highly related to other manifold learning
methods such as Laplacian eigenmaps and spectral clustering.  In fact, if data
is uniformly sampled from the underlying manifold, diffusion maps
\cite{Coifman:2006} is essentially eigenvalue weighted Laplacian eigenmaps
\cite{Talmon:2013}. 

Although we assume points on the intrinsic manifold are fixed, their geometry,
i.e, dependence structure, is allowed change. We can conceptualize samples from
a manifold with a time-dependent metric by considering a corresponding point
cloud smoothly moving through $\R^d$ produced by isometrically embedding the
manifold over time. The evolution of the metric dictates the movement of
points, and vice versa. In practice, datasets conforming to this model are
commonly encountered, e.g., an RGB video feed consists of a collection of $n$
pixels which move through $\R^3$. 

In general, we consider data consisting of a collection of $n$ points in $\R^d$
measured at $m$ times encoded in an $n \times d \times m$ data tensor $X$. The
tensor $X$ can be expressed as a sequence $(X_1,\ldots,X_m)$ of $n \times d$
matrices whose entries correspond across the sequence.  Given such as sequence
$(X_1,\ldots,X_m)$, the time coupled diffusion map framework introduced in this
paper is based on the product operator:
\begin{equation} \label{eqn: product operator}
\PP^{(m)} = \PP_m \PP_{m-1} \cdots \PP_2 \PP_1,
\end{equation}
where each $\PP_i$ is a diffusion operator constructed from $X_i$. We will show
that this discrete diffusion process, which is formally defined in the following
section, approximates a continuous diffusion process on an assumed underlying
manifold with a time-dependent metric. Additionally, we introduce the
notion of time coupled diffusion maps, named thus because the time
evolution of the data has been coupled to the time evolution of a
diffusion process.
 
\subsection{Related works}

In the diffusion geometry literature, several techniques have been developed,
which also utilize multiple diffusion kernels for a variety of objectives
including: iteratively refining the representation of data, facilitating
comparison, and combing multiple measurements of a fixed system.

An early example of a multiple kernel method is the denosing algorithm of Szlam,
Maggioni, and Coifman \cite{Szlam:2008}, which iteratively smooths an image via
an anisotropic diffusion process. That is, the algorithm switches between
constructing a diffusion kernel on a given data set (in this case an image), and
applying the constructed kernel to the data: 
$$
X_i \rightarrow \PP_i, \quad X_{i+1} = \PP_i X_i
$$
where the arrow denotes that $\PP_i$ is constructed based on $X_i$. More
recently, in \cite{welp:condensationSingleCell2016} Welp, Wolf, Hirn, and
Krishnaswamy introduce an iterative diffusion based construction, which acts to
course grain data. From a theoretical perspective, both of these methods can be
considered in the context of the time coupled diffusion framework introduced in
this paper.

In \cite{Wang:2012} Wang, Jiang, Wang, Zhou, and Tu introduce the notion of
Cross Diffusion as a metric fusion algorithm with applications to image
processing. They demonstrate how multiple metrics on a given data set
can be combined by considering the iterative cross diffusion
\[
\PP^{(t+1)}_1 = \PP_1 \PP_2^{(t)} \PP_1^T, \quad \text{and} \quad
\PP^{(t+1)}_2 = \PP_2 \PP_1^{(t)} \PP_2^T, 
\]
where $\PP_1^{(0)} = \PP_1$ and $\PP_2^{(0)} = \PP_2$ are constructed from two different metrics on the
given data. A generalized method for $m$ metrics is also described.

In \cite{Hirn:2014} Coifman and Hirn present a method of extending diffusion
maps to allow comparisons across multiple measurements of a system, even when
such measurements are of different modalities. 

In \cite{Lederman:2014, Lederman:2015b}
Lederman and Talmon introduce the idea of Alternating Diffusion: a method of
combining measurements from multiple sensors to exact the common source of
variability (i.e., the common manifold), while filtering out sensor specific
effects. The method is based on the product operators 
\[
\PP_1 \PP_2, \quad \text{and} \quad \PP_2 \PP_1,
\]
where $\PP_1$ and $\PP_2$ are constructed from two different views of the data.
In \cite{Lindenbaum:multiviewDiff2015}, Lindenbaum, Yeredor, Salhov, and
Averbuch follow a similar approach, but concatenate a collection of alternating
products in block matrix defining Multi-View Diffusion Maps.  Recently, several
applications and extensions of Alternating Diffusion have been developed. In
\cite{Lederman:2015b} Lederman, Talmon, Wu, Lo, and Coifman demonstrate an
application of Alternating Diffusion to sleep stage assessment.  In
\cite{Talmon:2016} Talmon and Wu describe a general notion of nonlinear manifold
filtering, which extracts a common manifold from multiple sensors.

Our work extends the current diffusion maps literature by considering evolving
dynamics rather than enhancing the analysis of a fixed system. We consider an $n
\times d \times m$ data tensor describing a system of $n$ points in $\bR^d$ over
$m$ times, and seek to construct a manifold model and diffusion geometry
framework for this setting. A similar framework is considered by Banisch and
Koltai in \cite{banisch:dmTrajectory2016}. However, rather than study the
product operator \eqref{eqn: product operator}, they study the sum of the
operators $\PP_i$ and prove a relation with the dynamic Laplacian.

\subsection{Organization}
The remainder of the paper is organized as follows. In Section \ref{sec:tcdm} we
describe the construction of time coupled diffusion maps. In Section
\ref{sec:manifold} we establish a connection between the product operator
$\PP^{(t)}$ and the heat kernel on the assumed underlying manifold with a
time-dependent metric. In Section \ref{sec:numerical} we present numerical
results on synthetic data. Section \ref{sec: cont tcdm} investigates the
continuous analog of the operator $\PP^{(t)}$ and the corresponding time coupled
diffusion distance. Concluding remarks are given in Section
\ref{sec: conclusion}. 

\section{Time coupled diffusion maps} \label{sec:tcdm}
In this section, we introduce the notion of time coupled diffusion maps. 

\subsection{Notation} \label{sec:2.2} Let $\M$ be a compact smooth manifold with
a smooth $1$-parameter family of Riemannian metrics $g(\tau)$, $\tau \in [0,T]$,
$T < \infty$. We refer to the parameter $\tau$ as time throughout. For each time
$\tau$, let $\iota_{\tau}: \M \hookrightarrow \R^d$, $\M_{\tau} =
\iota_{\tau}(\M)$, denote an isometric embedding of $(\M,g(\tau))$ into
$d$-dimensional Euclidean space, where $d$ is fixed for all $\tau \in [0,T]$.
Let $X = \{x_j\}_{j=1}^n \subseteq \M$ denote a finite collection of $n$ points
sampled from $\M$ and let $(\tau_0,\tau_1,\ldots,\tau_m)$ be a uniform partition
of $[0,T]$ with $\tau_0 = 0$ and $\tau_m = T$.  We assume that our data consists
of measurements of the $n$ points $X$ at times $\tau_1,\ldots,\tau_m$ such that
we have one measurement set for each of the time intervals $(\tau_{i-1},\tau_i]$
for $i=1,\ldots,m$.  More precisely, our data will consist of a sequence of $m$
sets $(X_1,\ldots, X_m)$, where $X_i= \iota_{\tau_i}(X) =
\{x_j^{(i)}\}_{j=1}^n$; see Figure \ref{fig: data model} for an illustration. 
%
%
\begin{figure}[h]
\centering
\begin{tabular}{ccc}
$\R^d$ &
$\tau_1,\tau_2,\tau_3,\ldots,\tau_m$ &
$\big(\cM,g(\tau) \big)$ \\
\begin{minipage}{2.2in}
\[
\left.
\begin{array}{ccccc}
\includegraphics[scale= .08]{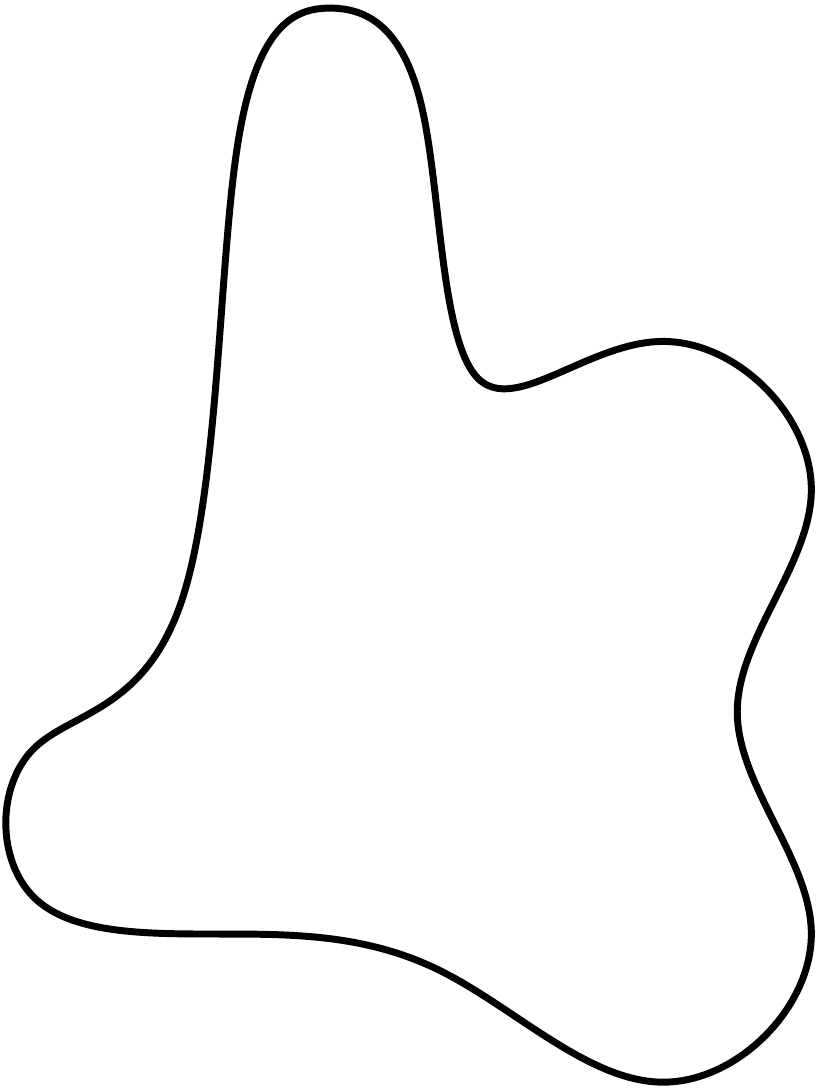} &
\includegraphics[scale = .08]{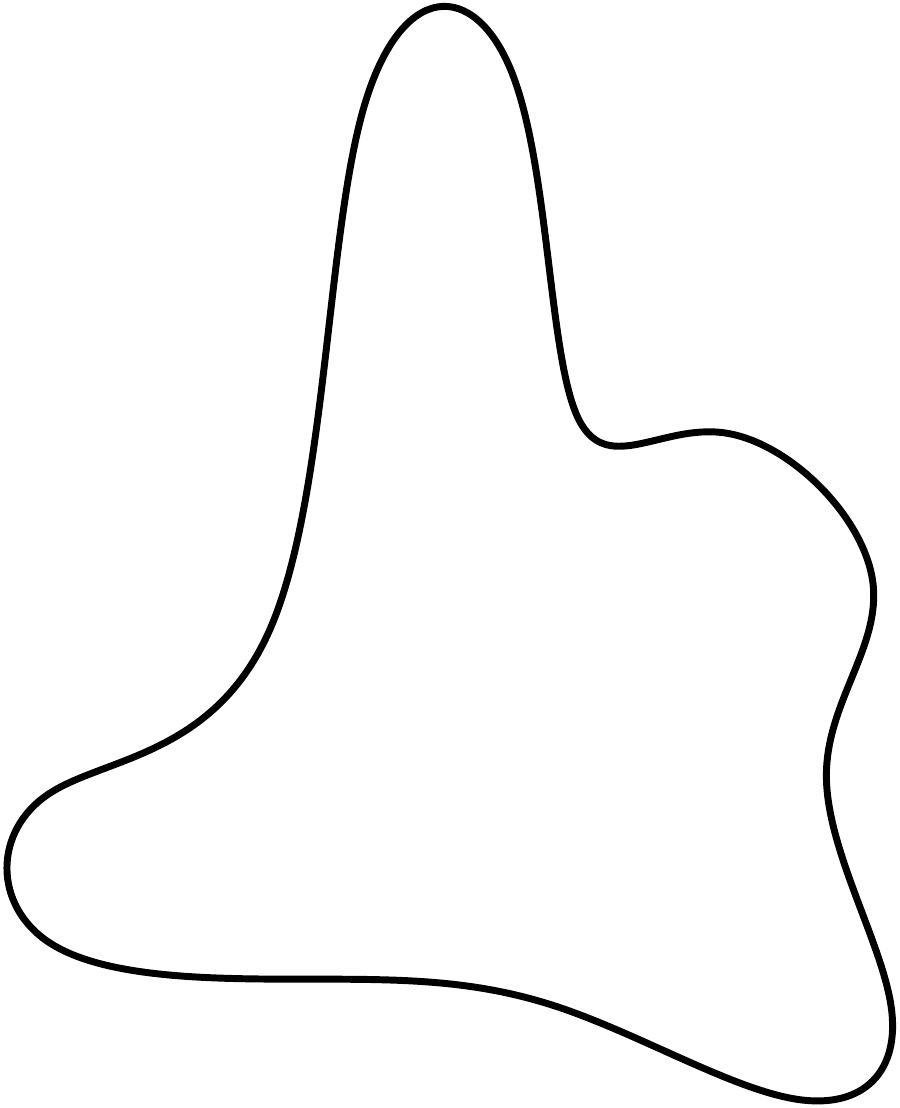} &
\includegraphics[scale = .08]{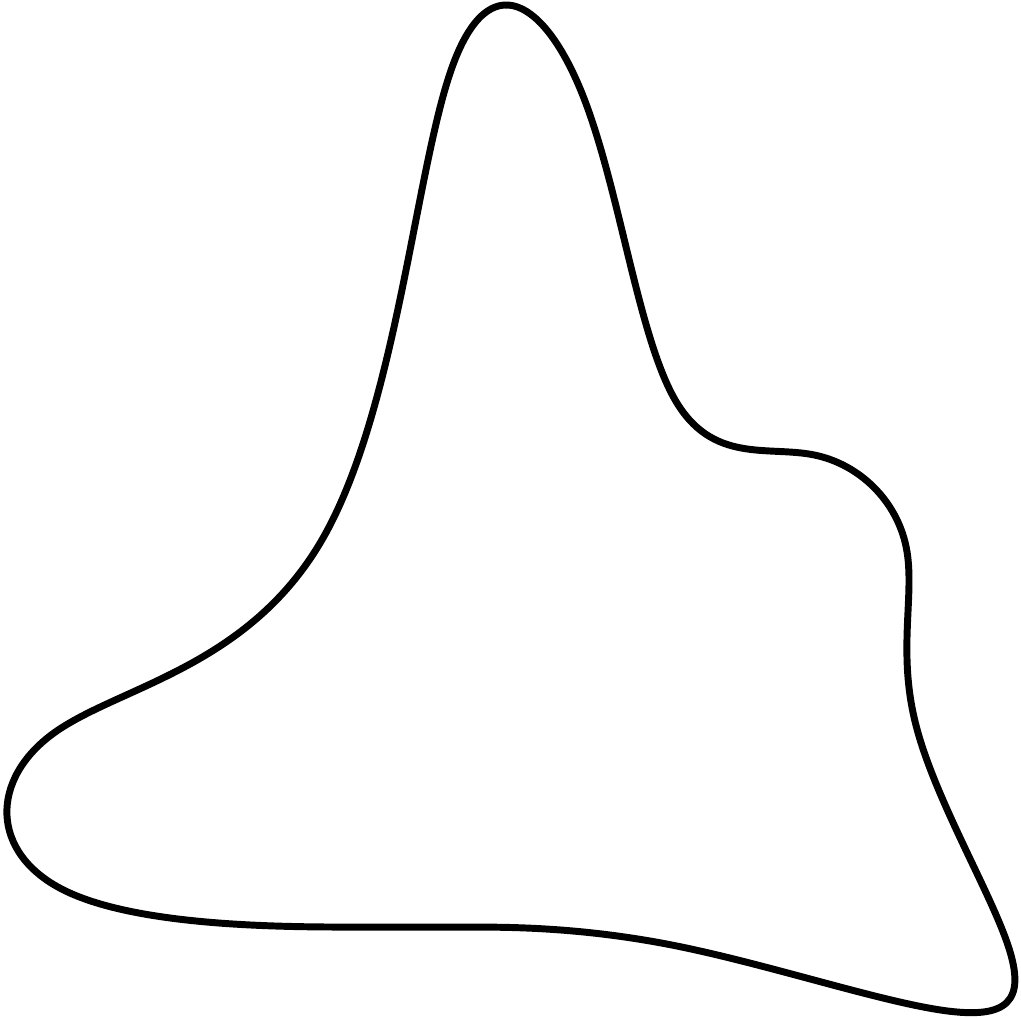} &
{ \cdots }&
\includegraphics[scale = .08]{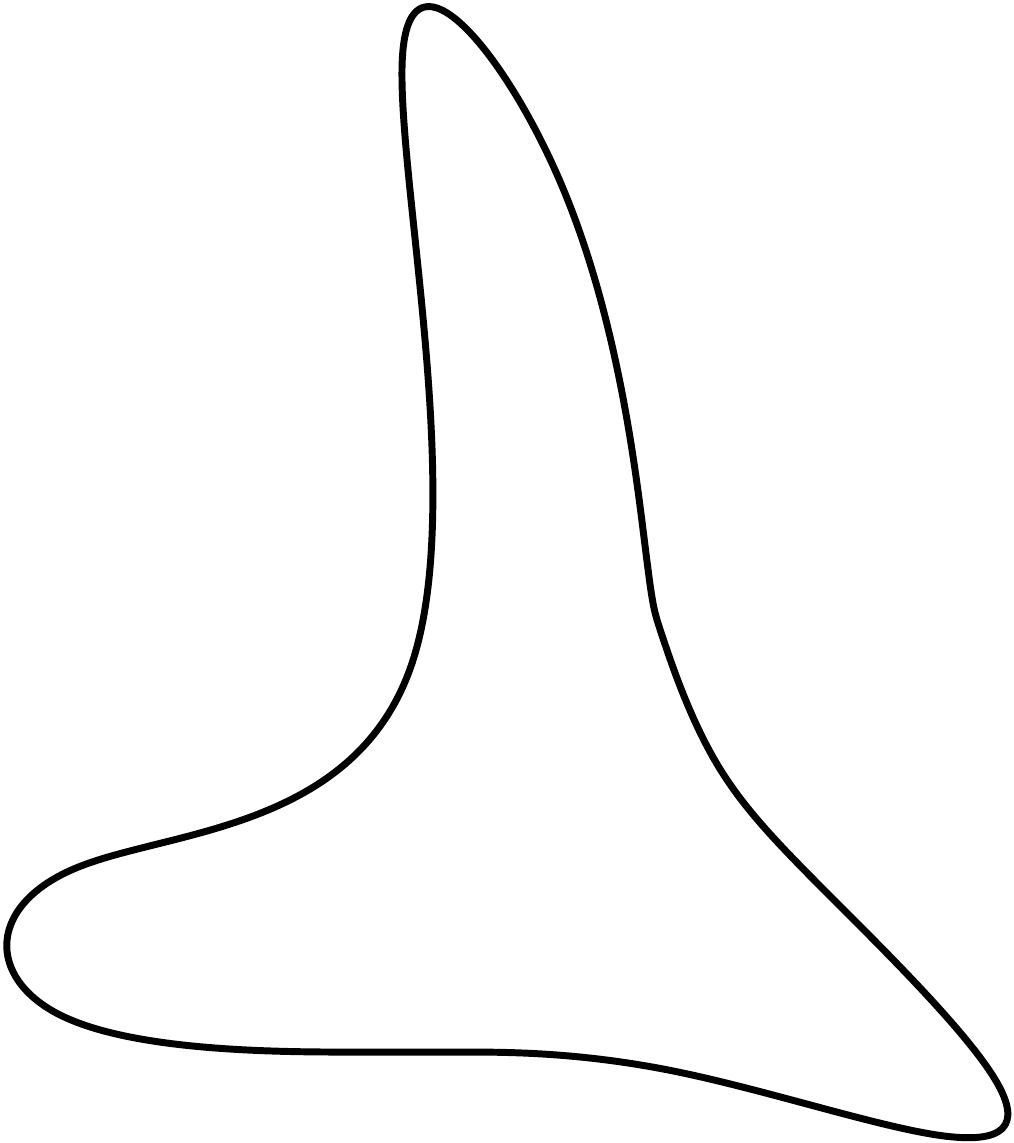} \\
X_1 & X_2 & X_3 & & X_m \\
\end{array} 
\right\}
\]
\end{minipage} &
\begin{minipage}{.9in}
\[ \includegraphics[width = .4in]{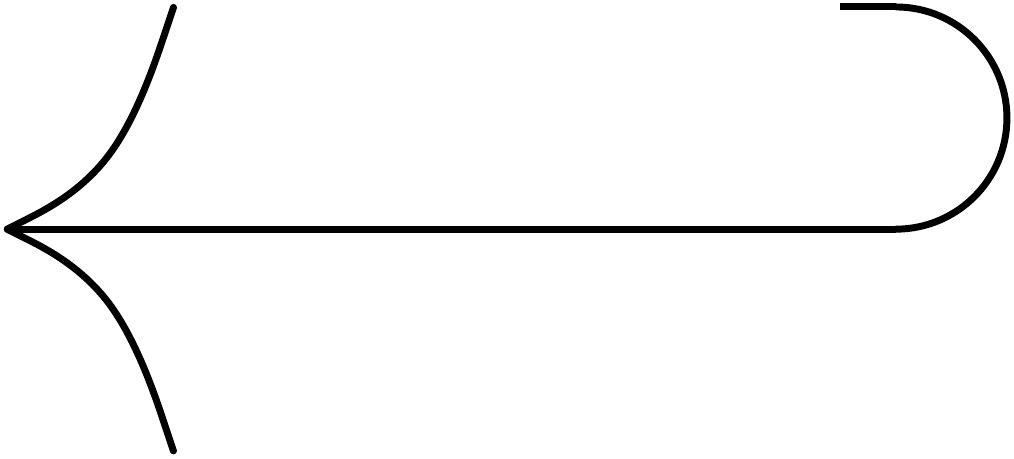} \]
\end{minipage}
&
\begin{minipage}{1in}
\[ \vcenter{\hbox{\includegraphics[scale = .14]{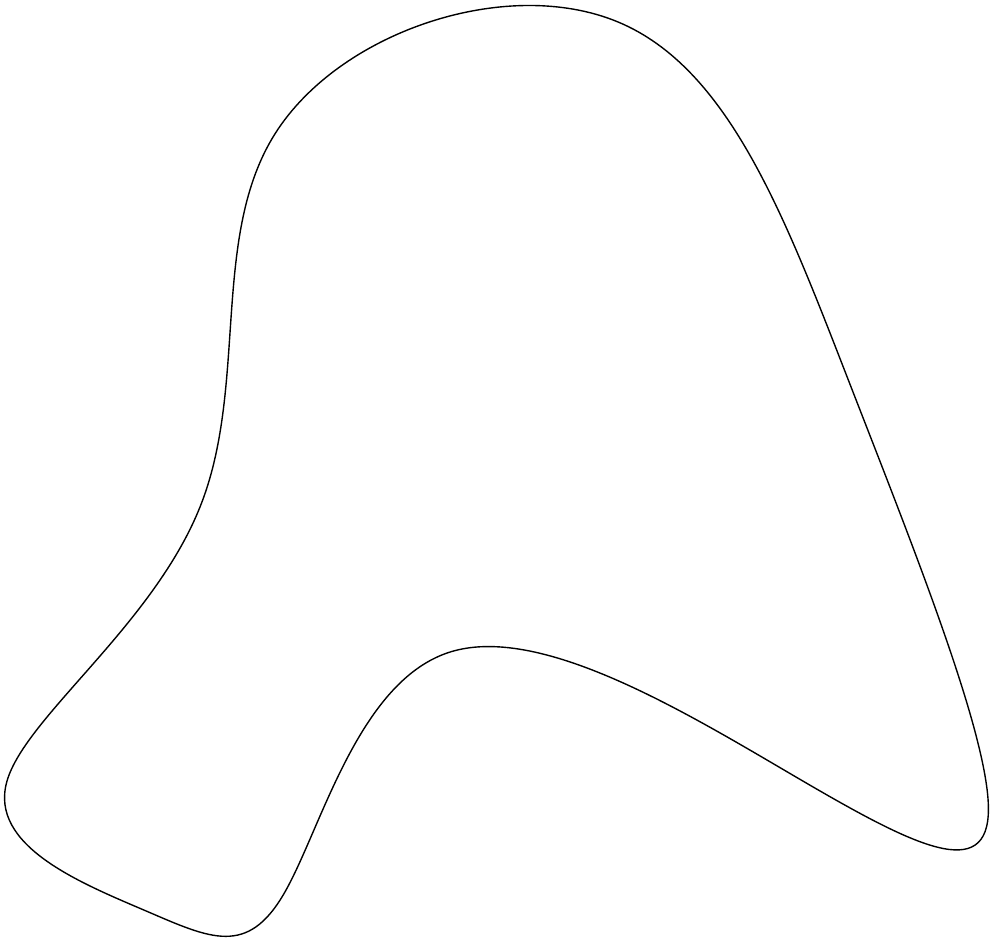}}}\] 
\end{minipage} \\
\end{tabular}
\caption{Illustration of the data model}
\label{fig: data model}
\end{figure}

Such data can be represented as an $n \times d \times m$ tensor corresponding to
$n$ points in $\R^d$ measured at $m$ times. Suppose each $X_i$ is distributed
over $\M_{\tau_i}$ according to a density $q_{\tau_i}: \M_{\tau_i} \rightarrow
\R$. We assume that at time $\tau_1$, the $n$ points $X_1 =
\{x_j^{(1)}\}_{j=1}^n$ are sampled independently
from $\M_{\tau_1}$ according to the density $q_{\tau_1}$. Since no
re-sampling occurs, any changes in the densities $q_{\tau_i}$, for
$i=1,\ldots,m$, result from deformations of $q_{\tau_1}$ induced by changes in
the Riemannian volume of $\M$, which itself is induced from changes in the
Riemannian metric $g(\tau)$ over time. In particular, $x_j^{(i)} =
\iota_{\tau_i} \circ \iota_{\tau_1}^{-1} (x_j^{(1)})$ and $q_{\tau_i}(x) =
q_{\tau_1}(\iota_{\tau_1} \circ \iota_{\tau_i}^{-1}(x)) | \det(D(\iota_{\tau_1}
\circ \iota_{\tau_i}^{-1}))(x)|$.  

\begin{rmk*} Two remarks are in order. 
\begin{enumerate}
\item Even if the initial sampling density is uniform, the changing geometry may
alter the density over time. Hence, the density invariant kernel construction
described in \cite{Lafon:2004, Coifman:2006} plays an essential role in our
construction. We assume the minimum sampling density over time is bounded
below and the associated error term enters our error analysis as a constant. We
refer the reader to a recent paper by Berry and Harlim \cite{Berry:2015} for a
detailed error analysis of variable density kernel constructions. 
\item Since we have assumed the initial samples $X_1$ are i.i.d., then for fixed
$i$, each set of samples $X_i$ are also i.i.d since $X_i$ is the continuous,
hence measurable, function $\iota_{\tau_i} \circ \iota^{-1}_{\tau_1}$ of the
i.i.d. variables $X_1$.
\end{enumerate}
\end{rmk*}

\subsection{Kernel construction}
Given a sequence of data sets $(X_1,\ldots,X_m)$ as described above, we proceed
as follows. For each data set $X_i$, we construct a diffusion operator
$\PP_{\varepsilon,i}$ following the diffusion maps framework
\cite{Coifman:2006}. For completeness, we include the details of the
construction of $\PP_{\varepsilon,i}$ in the following. For each time index $i$,
we define a Gaussian kernel $\K_{\varepsilon,i}$ on $X$ using the measurements
$X_i$,
\begin{equation} \label{eqn:markov1} 
\K_{\varepsilon, i}(x_j,x_k) =
\exp\left( - \frac{\|x_j^{(i)}-x_k^{(i)}\|^2_2}{4\varepsilon} \right),
\quad x_j, x_k \in X \text{ and } x_j^{(i)}, x_k^{(i)} \in X_i,
\end{equation} 
where $\|\cdot\|_2$ denotes the Euclidean norm in the ambient space
$\bR^d$. By summing over the second variable of the kernel
$\K_{\varepsilon,i}$ we approximate the density $q_{\tau_i}$ by
\begin{equation} \label{eqn:markov2} 
\q_{\varepsilon, i}(x_j) =
\sum_{k=1}^n \K_{\varepsilon,i}(x_j,x_k), 
\quad x_j \in X.
\end{equation} 
Using $\q_{\varepsilon,i}$ we normalize the kernel
$\K_{\varepsilon,i}$ as follows
\begin{equation} \label{eqn:markov3} 
\widetilde{\K}_{\varepsilon, i}(x_j,x_k) = 
\frac{\K_{\varepsilon, i}(x_j, x_k)}{\q_{\varepsilon, i} (x_j) \q_{\varepsilon,i}(x_k)},
\quad x_j, x_k \in X.  
\end{equation} 
Then the corresponding diffusion operator $\PP_{\varepsilon,i} : \R^n
\rightarrow \R^n$ is given by
\begin{equation} \label{eqn:markov4}
(\PP_{\varepsilon,i} f)(x_j) = 
\sum_{k=1}^n \frac{\widetilde{\K}_{\varepsilon,i}(x_j,x_k)}{
\sum_{l=1}^n \widetilde{\K}_{\varepsilon,i}(x_j,x_l)} f(x_k).
\end{equation}
Each $\PP_{\varepsilon,i}$ is a matrix which is a diffusion operator
when applied to column vectors on the left, and a Markov operator when
applied row vectors on the right. For $1 \le t \le m$, define:
\begin{equation} \label{eqn:Pt} \PP_{\varepsilon}^{(t)} =
\PP_{\varepsilon,t} \PP_{\varepsilon,t-1} \cdots \PP_{\varepsilon,1}.
\end{equation}
Like its constituent components, the matrix $\PP_{\varepsilon}^{(t)}$ is
a diffusion operator when acting on column vectors from the left and a
Markov operator when acting on row vectors from the right. As
a Markov operator, $\PP_{\varepsilon}^{(t)}$ acts backwards in time, in
the sense that $\PP_{\varepsilon}^{(t)}$ acts by first applying
$\PP_{\varepsilon,t}$, second $\PP_{\varepsilon,t-1}$ and so on.
On the other hand, as a diffusion operator $\PP_{\varepsilon}^{(t)}$ acts
forward in time, first applying $\PP_{\varepsilon,1}$, second
applying $\PP_{\varepsilon,2}$ and so on. We have chosen the
ordering convention in equation \eqref{eqn:Pt} to favor the geometric
interpretation. 

\subsection{Time coupled diffusion distance}
Let $\vdelta_j$ denote a Dirac distribution centered at $x_j$, i.e.,
$\vdelta_j(x_j) = 1$ and $\vdelta_j=0$ elsewhere.  We compare the points $x_j$
and $x_k$ by comparing the posterior distributions of $\vdelta_j^T$ and
$\vdelta_k^T$ under the Markov operator $\PP_{\varepsilon}^{(t)}$. More
specifically, following \cite{Coifman:2006} we define a diffusion based distance
as the $L^2$ distance between these posterior distributions weighted by the
reciprocal of the stationary distribution of the Markov chain. That is, we
define the distance $D^{(t)}_\varepsilon$ by 
\begin{equation} \label{eqn: tcdd}
D_\varepsilon^{(t)}(x_j,x_k) = \| \vdelta_j^T \PP_{\varepsilon}^{(t)}
- \vdelta_k^T \PP_{\varepsilon}^{(t)} \|_{L^2(1/\vpi_{(t)})},
\end{equation}
where $\vpi_{(t)}$ is the stationary distribution of
$\PP_{\varepsilon}^{(t)}$, i.e., $\vpi_{(t)}^T \PP_{\varepsilon}^{(t)} =
\vpi_{(t)}^T$, and $\|\cdot\|_{L^2(1/\vpi_{(t)})}$ is the weighted $L^2$ norm:
\begin{equation} \label{eq:weightl2}
\|\vf \|_{L^2(1/\vpi_{(t)})} \defeq \sqrt{ \sum_{j=1}^n \vf(x_j)^2
\frac{1}{\vpi_{(t)}(x_j)}}.
\end{equation}
We refer to $D_{\varepsilon}^{(t)}$ as the \textit{time coupled diffusion
distance}, since it extends the diffusion distance in \cite{Coifman:2006} by
coupling the evolution time of the data $(X_1,\ldots,X_t)$ with that of the
diffusion process governed by
$(\PP_{\varepsilon,1},\ldots,\PP_{\varepsilon,t})$. Our next objective is to
construct a \textit{time coupled diffusion map}, that is, an embedding of $X$
into Euclidean space which preserves the time coupled diffusion distance. For
notational brevity, we suppress the dependence on $\varepsilon$  in the
following. 

\subsection{Time coupled diffusion map}
Recall that the diffusion map for a static manifold $(\M, g)$ is defined in
terms of the eigenvectors and eigenvalues of the transition matrix $\PP$, which
is constructed in the same manner as \eqref{eqn:markov4}. The inhomogeneous
transition operator $\PP^{(t)}$ is also row stochastic, but unlike $\PP$,
does not necessarily have a complete basis of real eigenvectors. Instead, 
we define the operator $\vA^{(t)}$ by
\begin{equation} \label{eqn:PiP}
\vA^{(t)} = \vPi^{1/2}_{(t)} \PP^{(t)} \vPi_{(t)}^{-1/2},
\end{equation}
where $\vPi_{(t)}$ denotes the matrix with the stationary distribution
$\vpi_{(t)}$ of $\PP^{(t)}$  along the diagonal and zeros elsewhere. First,
observe that $\vpi^{1/2}_{(t)}$ is both a right and left eigenvector of
$\vA^{(t)}$ with eigenvalue $1$. In fact, $\vA^{(t)}$ has operator norm one and
naturally arises when constructing a diffusion framework starting with a Markov
chain, see \ref{dmfrommarkov}. Next, we compute the singular value
decomposition (SVD) of $\vA^{(t)}$:
\begin{equation} \label{eqn:Asvd}
\vA^{(t)} = \vU_{(t)} \vSigma_{(t)} \vV^T_{(t)},
\end{equation}
where $\vU_{(t)}$ is an orthogonal matrix of left singular vectors,
$\vSigma_{(t)}$ is a diagonal matrix of corresponding singular values, and
$\vV_{(t)}$ is an orthogonal matrix of right singular vectors. Define 
\begin{equation} \label{eqn:Psi}
 \vPsi^{(t)} \defeq \vPi_{(t)}^{-1/2} \vU_{(t)} \vSigma_{(t)}.
\end{equation}
%
%
\begin{lem} \label{lem: embedding} The embedding 
\begin{equation} \label{eqn:tcdm}
x_j \mapsto \vdelta_j^T \vPsi^{(t)}
\end{equation}
of the data $X$ into Euclidean space preserves the time coupled diffusion
distance. That is to say,
\[
D^{(t)}(x_j,x_k) = \| \vdelta_j^T \vPsi^{(t)} - \vdelta_k^T \vPsi^{(t)}
\|_{L^2}.
\]
We refer to the embedding  $x_j \mapsto \vdelta_j^T \vPsi^{(t)}$  as the
time coupled diffusion map.  
\end{lem}
%
%
\begin{proof} By the definition of time coupled diffusion distance \eqref{eqn: tcdd}, and by definition of the weighted $L^2$ norm \eqref{eq:weightl2} 
\begin{equation} \label{eq:step1}
D^{(t)}(x_j,x_k) = \| \vdelta_j^T \PP^{(t)} - \vdelta_k^T \PP^{(t)}
\|_{L^2(1/\vpi_{(t)})} = \| \vdelta_j^T \PP^{(t)} \vPi^{-1/2} - \vdelta_k^T
\PP^{(t)} \vPi^{-1/2}
\|_{L^2}.
\end{equation}
Multiplying equation \eqref{eqn:PiP} by $\vPi_{(t)}^{-1/2}$ yields 
\[
\vPi^{-1/2}_{(t)} \vA^{(t)} =  \PP^{(t)} \vPi_{(t)}^{-1/2},
\]
and substituting this expression into \eqref{eq:step1} gives
\[
D^{(t)}(x_j,x_k) = 
\| \vdelta_j^T \vPi^{-1/2}_{(t)} \vA^{(t)} - \vdelta_k^T  \vPi^{-1/2}_{(t)}
\vA^{(t)} \|_{L^2}.  \]
Expanding $\vA^{(t)}$ in its singular value decomposition \eqref{eqn:Asvd}
yields,
\[
D^{(t)}(x_j,x_k) = 
\| \vdelta_j^T \vPi^{-1/2}_{(t)} \vU_{(t)} \vSigma_{(t)} \vV^T_{(t)} - \vdelta_k^T \vPi^{-1/2}_{(t)} \vU_{(t)} \vSigma_{(t)} \vV^T_{(t)} \|_{L^2}.
\]
Now, since the transformation $\vV_{(t)}$ is orthogonal,
\[
D^{(t)}(x_j,x_k)  = \| \vdelta_j^T \vPi^{-1/2}_{(t)} \vU_{(t)} \vSigma_{(t)} -
\vdelta_k^T \vPi^{-1/2}_{(t)} \vU_{(t)} \vSigma_{(t)} \|_{L^2},
\]
and substituting $\vPsi^{(t)} = \vPi_{(t)}^{-1/2} \vU_{(t)} \vSigma_{(t)}$ into
this equation yields the result.
\end{proof}

\begin{rmk*} Several remarks regarding time coupled diffusion maps are in order.
\begin{enumerate}
\item If $t=1$, the time coupled diffusion map \eqref{eqn:tcdm} is equivalent to
the definition of the standard density invariant diffusion map in
\cite{Coifman:2006}, see the calculations in \ref{dmfrommarkov}, which are
motivated by similar calculations in \cite{Lafon:2004}.  
\item The first coordinate of the embedding \eqref{eqn:tcdm} is always constant
because $\vpi_{(t)}^{1/2}$ is the top left singular vector of $\vA^{(t)}$, and
therefore, this coordinate of the embedding can be discarded.  
\item In order to produce an embedding of $X$ into $\bR^l$ for some $l > 0$, we
can compute a rank $l+1$ singular value decomposition in equation
\eqref{eqn:Asvd} and map
\begin{equation*}
x_j \mapsto \vdelta_j^T \vPsi^{(t)}_l,
\end{equation*}
where $\vPsi^{(t)}_l$ denotes the
matrix consisting of columns $2$ through $l+1$ of the matrix $\vPsi^{(t)}$.
Since the singular values of $\vA^{(t)}$ are of the form $1 = \sigma_0 >
\sigma_1 \ge \cdots \ge \sigma_{n-1}$, the truncated embedding will preserve the
diffusion distance up to an error on the order of $\sigma_{l+1}$. As in other 
diffusion based methods, in the case where data lies on a low dimensional
manifold, we expect the first few coordinates of the embedding to provide a
meaningful summary of the data.
\end{enumerate}
\end{rmk*}

\subsection{Comparison to standard diffusion maps}
It is instructive at this point to make a comparison with the original diffusion
maps of Coifman and Lafon \cite{Coifman:2006}. In that setting, one is given
samples $X$ of an isometric embedding of a manifold $\M$ with a \emph{static}
metric $g$.  The density invariant version of diffusion maps uses the same
construction as outlined in equations \eqref{eqn:markov1}, \eqref{eqn:markov2},
\eqref{eqn:markov3}, and \eqref{eqn:markov4}, which yields a Markov matrix
$\PP$, which is not indexed by a time $i$ since the metric is static.
Running this homogeneous Markov chain forward $t$ steps is equivalent to
composing $\PP$ with itself $t$ times, i.e., $\PP^t$. Since $\PP$ is, by
construction, similar to a symmetric matrix, it can be decomposed into a
diagonal matrix of real eigenvalues $\vLambda$ and a basis of right eigenvectors
$\vPhi$. Since the left eigenvectors of $\PP$ are orthogonal in $L^2(1/\vpi)$,
the diffusion map
\begin{equation} \label{eqn: original diffusion maps}
x_j \mapsto \vdelta_j^T \vPhi \vLambda^t,
\end{equation}
preserves the diffusion distance $D^t(x_j, x_k) = \| \vdelta_j^T \PP^t -
\vdelta_k^T \PP^t \|_{L^2(1/\vpi)}$, where $\vpi$ is the stationary distribution
of $\PP$. A main result of \cite{Coifman:2006} is that as $n \rightarrow \infty$
and $\varepsilon \rightarrow 0$, we have $\PP_{\varepsilon}^{t/\varepsilon}
\rightarrow e^{t\Delta}$, where $e^{t\Delta}$ is the Neumann heat kernel on
$\M$.  Since the eigenfunctions and eigenvalues of $e^{t\Delta}$ give a complete
geometric description of the manifold $\M$
\cite{berard:embedManifoldHeatKer1994}, the diffusion maps embedding \eqref{eqn:
original diffusion maps} learns the geometry of the manifold from the samples
$X$.

In the time-dependent case, we consider a sequence of data
$(X_1,\ldots,X_t)$ and construct a corresponding family of operators
$(\PP_1,\ldots,\PP_t)$. Rather than raising a matrix to a power, we
compose the family of operators, obtaining $\PP^{(t)}$ defined in
equation \eqref{eqn:Pt}. After defining
a distance $D^{(t)}$ based on $\PP^{(t)}$ we seek a distance preserving
embedding. However, since the product of symmetric matrices is not in
general symmetric, there is no reason we should suspect
$\PP^{(t)}$ to have a basis of real eigenvalues and eigenvectors. Therefore, we
construct a diffusion maps framework based on the Markov operator
using the singular value decomposition of the
operator $\vA^{(t)}$ as described in equations 
\eqref{eqn:PiP}, \eqref{eqn:Asvd}, \eqref{eqn:Psi}, and \eqref{eqn:tcdm} . We
expect the operator $\PP^{(t)}$
encodes some average sense of affinity across the data, and by extension 
that the time coupled diffusion map is similarly meaningful. In
the following, we make the connection of $\PP^{(t)}$ to the data precise
by showing that
$\PP^{(t)}$ approximates the heat kernel on the
underlying manifold $(\M,g(\cdot))$, which will be
defined in Section \ref{sec:manifold}. Thus, we conjecture that the
time coupled diffusion map aggregates important geometrical
information of the manifold $(\M,g(\cdot))$ over the time interval
$[0,T]$. Numerical results in Section \ref{sec:numerical} lend
credence to this conjecture. There is, however, no theoretical result
that directly links geometrical information of $(\M,g(\cdot))$ over
arbitrarily long time scales with its heat kernel. The closest results
are contained in \cite{Abdallah:2010, abdallah:embedTimeManifold2012},
in which the heat kernel of $(\M,g(\cdot))$ is used to embed the
manifold into a single Hilbert space, so that one can observe the flow
of $\M$ over time. This result, however, only holds for short time
scales.

Before describing the connection of $\PP^{(t)}$ to heat flow, we make
a brief computational note. In practice, it is common to construct
each diffusion matrix $\PP_{i}$ defined in equation
\eqref{eqn:markov4} as a sparse matrix by, for example, truncating
values which fall below a certain threshold. However, if the product
\begin{equation*}
\PP^{(t)} = \PP_t \PP_{t-1} \cdots \PP_{2} \PP_1
\end{equation*}
is computed explicitly, there is no reason to expect sparsity will be
maintained. Moreover, since each $\PP_i$ is a diffusion operator on an
assumed underlying $k$-dimensional manifold $\cM$, if each $\PP_i$
initially has $l$ nonzero entires in each row, the number of nonzero
entires in each row of the product $\PP^{(t)}$ could be on the order
of $(l \cdot t)^k$. Therefore,
rather than explicitly constructing $\PP^{(t)}$, we can treat
$\PP^{(t)}$ as an operator that can be applied to vectors in order $n \cdot l
\cdot t$ operations. Using this approach, the singular value decomposition in
equation \eqref{eqn:Asvd} can be computed by
iterative methods such as Lancroz iteration or subspace iteration
\cite{Rutishauser:1969}. When $n^2 > n \cdot l \cdot t$ this
approach may provide significant computational savings, and reduce the memory
requirements of the computation.

\section{The heat kernel for $(\M, g(\cdot))$ \label{sec:manifold}}

We begin by stating the definition and properties of the heat kernel
for a manifold with time-dependent metric. Let $g(\tau), \tau \in
[0,T]$, be a smooth family of metrics on a manifold $\M$. The heat
equation for such a manifold is: 
\begin{equation} \label{eqn: heat eqn}
\frac{\partial
u}{\partial t} = \Delta_{g(t)} u, 
\end{equation} 
where $u: \M \times [0,T] \rightarrow \R$. We say that 
\begin{equation*} 
Z : (\M \times [0,T]) \times (\M \times [0,T]) \rightarrow \R,
\end{equation*} 
is the \textit{heat kernel} for $\frac{\partial}{\partial t} - \Delta_{g(t)}$
if the following three conditions hold: 
\begin{enumerate}
\label{eq:heat kernel}
\renewcommand{\theenumi}{$(H_{\arabic{enumi}})$}
\renewcommand{\labelenumi}{\theenumi} 
\item\label{h1} $Z(x,\tau;y,\sigma)$ is $C^2$ in the spacial variables $x,y$ and
$C^1$ in the temporal variables $\tau, \sigma$, 
\item\label{h2} $\left(\frac{\partial}{\partial t} - \Delta_{g(t)}\right)
Z(\cdot, \cdot; y,\sigma) = 0$, and
\item\label{h3}  $\lim_{\tau \searrow \sigma} Z(\cdot,\tau;y,\sigma) =
\delta_y$.
\end{enumerate} 
Several researchers have studied the heat equation \eqref{eqn: heat
  eqn} and corresponding heat kernel for a manifold $(\M,
g(\cdot))$, in large part due to its relationship with the Ricci flow
\cite{Guenther:2002, Chow:2008}, but also for data analysis
\cite{Abdallah:2010}. The following properties of the heat kernel have
been established (see \cite{Guenther:2002}). First, the heat kernel
exists and is the unique positive function satisfying \ref{h1},
\ref{h2}, and \ref{h3}. Additionally, the solution to the initial
value problem,
\begin{equation}\label{eqn: initial value}
\left\{
\begin{array}{rcl}
\partial u / \partial t&=&\Delta_{g(t)} u, \\[5pt]
u(x, 0)&=&f(x),
\end{array}
\right.
\end{equation}
has the integral representation \cite[Corollary 2.2]{Guenther:2002}:
\begin{equation*}
u(x,\tau) = \int_{\M} Z(x, \tau; y, 0) f(y) \, dV (y, 0),
\end{equation*}
where $V(\tau)$ is the Riemannian volume of $\M$ at time $\tau$. The heat kernel
$Z(x,\tau;y,\sigma)$ also possesses properties analogous to those of the
standard heat kernel, such as the semigroup property,
\begin{equation} \label{eq:semigroup} 
Z(x,\tau;y,\sigma) = \int_{\M}
Z(x,\tau;\xi,\nu) Z(\xi,\nu;y,\sigma) \, dV(\xi, \nu), \quad
\forall \, x,y \in \M, \, \nu \in (\sigma,\tau).
\end{equation}
Furthermore, its ``rows'' sum to one, precisely stated as 
\begin{equation} \label{eqn: row sums one}
\int_{\M}
Z(x,\tau;y,\sigma) \, dV(y, \sigma) = 1, \quad \forall \, x \in \M, \,
\sigma < \tau \in [0,T].
\end{equation}
In the case of a static metric, the associated integral transform of
the Neumann heat kernel is known to have the asymptotic approximation
\cite{Coifman:2006}:
\begin{equation} \label{eqn:expansion}
e^{t\Delta} = \lim_{\varepsilon \rightarrow 0} \left(I + \varepsilon
\Delta\right)^{t/\varepsilon} = \left(1+\varepsilon \Delta
\right)^{t/\varepsilon} + \bigo(\varepsilon).
\end{equation}
For $t \leq T$, we define the associated integral transform of $Z$ as the operator
$T_Z^{(t)}: L^2 (\M) \rightarrow L^2(\M)$,
\begin{equation} \label{eqn: TZt def}
T_Z^{(t)}f(x) = \int_{\M} Z(x, t; y, 0) f(y) \, dV(y,0).
\end{equation}
We want to express $T_Z^{(t)}$ in an analogous form to \eqref{eqn:expansion};
we would expect it to resemble $e^{\int_0^t \Delta_{g(\tau)}d\tau}$.
However, as the family $\Delta_{g(\tau)}$ is not in general
commutative with respect to composition, it is necessary to consider
the \emph{ordered exponential} of $\Delta_{g(\tau)}$ over $[0,t]$,
which is the equivalent of the exponential for non-communicative
operators, and which is denoted $\T e^{\int_0^t
\Delta_{g(\tau)}d\tau}$. The ordered exponential can be expressed in
terms of the power series: 

\begin{align*} 
\T  e^{\int_0^t \Delta_{g(\tau)}d\tau}  &= I + \int_0^t
\Delta_{g(\tau)} \, d\tau + \int_0^t \Delta_{g(\tau)} \int_0^{\tau}
\Delta_{g(\tau_1)} \, d\tau_1 \, d\tau +  \cdots \\ 
&= I + \sum_{i=1}^\infty \int_0^t
\Delta_{g(\tau_1)} \int_0^{\tau_1} \Delta_{g(\tau_2)}
\int_0^{\tau_2}\cdots \int_0^{\tau_{i-2}} \Delta_{g(\tau_{i-1})}
\int_0^{\tau_{i-1}} \Delta_{g(\tau_i)} \, d \tau_i \, d\tau_{i-1} \cdots
d\tau_2 \, d\tau_1, 
\end{align*} 
where $I$ denotes the identity in the
appropriate space. The operator $\T$ can be thought of as enforcing
time order in the products for the standard exponential expansion in a
sense which is equivalent to the power series definition, cf.
\cite{Giscard:2014}. Considering $\T e^{\int_0^t
\Delta_{g(\tau)}d\tau}$ as a formal power series, it is easy to see
that the heat equation is satisfied. Let   $\{\phi_l\}_{l \ge 0}$ be
the eigenfunctions of $\Delta_{g(0)}$, ordered in terms of increasing
eigenvalue, and define $E_K \defeq \Span \{\phi_l : 0 \le l \le K\}$.
Then $\Delta_{g(0)}$ is bounded on $E_K$, and furthermore as $g(\tau)$
is a smooth family of metrics for $\tau \in [0,T]$, we
can conclude $\Delta_{g(\tau)}$ is uniformly bounded on $E_K$ for all
$\tau \in [0,T]$ (see \ref{appendix: uniform bound}). Hence on
$E_K$ the power series for $\T e^{\int_0^t \Delta_{g(\tau)}d\tau}$
converges and by uniqueness we conclude $\T e^{\int_0^t
\Delta_{g(\tau)}d\tau}$ is the heat kernel on $E_K$. Henceforth we
will use 
\begin{equation} \label{eqn:hk} 
T_Z^{(t)} = \T e^{\int_0^t \Delta_{g(\tau)}d\tau}
\end{equation} 
as the notation for the Neumann
heat kernel. We remark that $E_K$
could also be defined independent of the differential structure of the
manifold as $K$-band-limited functions on $\M$, cf.
\cite{Mousazadeh:2015}. We prefer the former definition as it provides
consistency with the definitions in \cite{Coifman:2006}.

Recall that $n$ is the number of spatial samples of the manifold $\M$,
and that $m$ is the number of temporal measurements on
the time interval $[0,T]$. We assume that the time interval
$[0,T]$ is divided into $m$ intervals $[\tau_{i-1},\tau_i)$ each of
length $\varepsilon$ where $\tau_0 = 0$ and $\tau_m = T$. For
simplicity, we assume that our $m$ measurements are taken at
$(\tau_1,\ldots,\tau_m)$.  Our main
result is that in the limit of large data, both spatially and
temporally, the transition operator
$\PP_{\varepsilon}^{(\lceil t/\varepsilon \rceil)}$ converges to the heat kernel:
\begin{equation*}
\PP_{\varepsilon}^{(\lceil t/\varepsilon \rceil)} \rightarrow
\T e^{\int_0^t \Delta_{g(\tau)} d\tau} \text{ as } n \rightarrow
\infty \text{ and } \varepsilon \rightarrow 0.
\end{equation*}
More precisely: 
\begin{thm} \label{thm:main} Suppose the isometric
embedding $\M_\tau \subset \R^d$ of a  time-dependent manifold $(\M,
g(\tau))$ is measured at a common set $X = \{x_j\}_{j=1}^n \subset \M$
of $n$ points at $\varepsilon$ spaced units of time over a time
interval $[0,T]$, so that, in particular, we have time samples
$(\tau_i)_{i=1}^m \subset [0,T]$ with $\tau_i = i \cdot
\varepsilon$ and $m = T/\varepsilon$. 

Then, for any sufficiently smooth function $f: \M \rightarrow \R$ and
$t \leq T$, the heat kernel $\T e^{\int_0^t \Delta_{g(\tau)} d\tau}$ can be approximated
by the operator $\PP^{(\lceil t/\varepsilon \rceil)}_{\varepsilon}$: 
\begin{equation}
\PP_{\varepsilon}^{(\lceil t/\varepsilon \rceil )}f(x_j) = \T e^{\int_0^t
\Delta_{g(\tau)} d\tau}f(x_j) + \bigo\left(\frac{1}{n^{1/2}
\varepsilon^{d/4+1/2}},\varepsilon \right), \quad x_j \in X.
\end{equation}
\end{thm}

We prove Theorem \ref{thm:main} in the following.
Recall that $E_K$ is defined to be the span for the first $K+1$
eigenfunctions of
the Laplace-Beltrami operator $\Delta_{g(0)}$. 

\begin{lem} \label{prop: heat kernel 1st order expansion} 
On $E_K$, the heat kernel $\T e^{\int_{\sigma}^{\sigma + \varepsilon}
  \Delta_{g(\tau)} d\tau}$ admits the
following asymptotic expansion:
\begin{equation} \label{eqn:prop1}
\T e^{\int_{\sigma}^{\sigma + \varepsilon} \Delta_{g(\tau)} d\tau} = I +
\varepsilon \cdot \Delta_{g(\sigma)} + \bigo(\varepsilon^2).
\end{equation}
\end{lem}

\begin{proof} 
By definition 
\[
\T e^{\int_{\sigma}^{\sigma + \varepsilon} \Delta_{g(\tau)}d\tau} = \]
\[ I + \int_{\sigma}^{\sigma + \varepsilon} \Delta_{g(\tau)} d\tau +
\int_{\sigma}^{\sigma + \varepsilon} \Delta_{g(\tau)}
\int_{\sigma}^{\tau}  \Delta_{g({\tau_1})} d\tau_1 d\tau +
\int_{\sigma}^{\sigma + \varepsilon} \Delta_{g(\tau)}  \int_{\sigma}^{\tau}  \Delta_{g({\tau_1})}
\int_{\sigma}^{\tau_1}  \Delta_{g({\tau_2})} d\tau_2 d\tau_1 d\tau + \cdots
\]
The tail of this series starting with the third
term is $\bigo(\varepsilon^2)$. For the second term, we can separate the
contribution of $\Delta_{g(\sigma)}$ from the integral yielding:
\begin{equation*} 
\T e^{\int_{\sigma}^{\sigma + \varepsilon} \Delta_{g(\tau)}d\tau}  = I + \varepsilon
\cdot \Delta_{g(\sigma)} + \int_{\sigma}^{\sigma + \varepsilon}
(\Delta_{g(\tau)}-\Delta_{g(\sigma)}) d\tau + \bigo(\varepsilon^2).
\end{equation*}
Now we can bound the integral by using the smoothness of
$\Delta_{g(\tau)}$, which results from the smoothness of $g(\tau)$. By
the smoothness of $\Delta_{g(\tau)}$ with respect to time and the fact
that it is uniformly bounded on $E_K$ for all $\tau \in [0,T]$,
\begin{equation*}
\left\|\Delta_{g(\tau)}-\Delta_{g(\tau^\prime)}\right\| \le C
|\tau-\tau^\prime|, \quad C \in \R^+.
\end{equation*}
Therefore,
\begin{equation*} \left| \int_{\sigma}^{\sigma + \varepsilon}
    (\Delta_{g(\tau)}-\Delta_{g(\sigma)}) d\tau \right| \le
  \varepsilon \cdot \sup_{\tau \in [\sigma, \sigma + \varepsilon]}
  \left\| \Delta_{g(\tau)}-\Delta_{g(\sigma)} \right\| \le C \varepsilon^2,
\end{equation*}
and hence we have:
\begin{equation*} 
\T e^{\int_{\sigma}^{\sigma + \varepsilon} \Delta_{g(\tau)}
d\tau}  = I + \varepsilon \cdot \Delta_{g(\sigma)} + \bigo(\varepsilon^2),
\end{equation*}
as desired.
\end{proof}

By using the short time expansion from \eqref{eqn:prop1}, we can develop a
long time asymptotic expansion of the heat kernel. In the following,
we write $\Delta_i=\Delta_{g(\tau_i)}$ to simplify notation.

\begin{lem} \label{prop:approximation} Let
$(\tau_0,\tau_1,\ldots,\tau_m)$ be a uniform partition of $[0,T]$ with
step size $\varepsilon > 0$, so that $\tau_i = \varepsilon \cdot i$
and $\tau_m = m \cdot \varepsilon = T$. For $0 < t \leq T$, define $\ell
\defeq \min_i \{i = 1, \ldots, m \mid \tau_i \geq t \}$. Then on $E_K$ the
heat kernel $\T e^{\int_0^t \Delta_{g(\tau)} d\tau}$ has the expansion
\begin{equation} \label{eq:heat kernel expansion} 
\T e^{\int_0^t \Delta_{g(\tau)} d\tau}= \left(I+\varepsilon \Delta_{\ell} \right)
\left(I+\varepsilon \Delta_{\ell-1} \right) \cdots \left( I+\varepsilon
\Delta_{1} \right) + \bigo(\varepsilon).
\end{equation} 
\end{lem}

\begin{proof} 
First suppose $t = \tau_{\ell}$. Using
the semigroup property \eqref{eq:semigroup} of the heat kernel, we have: 
\begin{equation} \label{eqn: apply semigroup}
\T e^{\int_0^t \Delta_{g(\tau)} d\tau} = \T
e^{\int_{\tau_{\ell-1}}^{\tau_{\ell}} \Delta_{g(\tau)} d\tau} \cdot \T
e^{\int_{\tau_{\ell-2}}^{\tau_{\ell-1}} \Delta_{g(\tau)} d\tau} \cdots \T
e^{\int_{\tau_0}^{\tau_1} \Delta_{g(\tau)} d\tau}.  
\end{equation}
Combining the expansion \eqref{eqn: apply semigroup} and Lemma
\ref{prop: heat kernel 1st order expansion}, yields the following:
\begin{align*}
\T e^{\int_0^t \Delta_{g(\tau)} d\tau} &= \left(I+\varepsilon \Delta_{\ell} +\bigo(\varepsilon^2)\right)
\left(I+\varepsilon \Delta_{\ell-1} + \bigo(\varepsilon^2) \right) \cdots
\left( I-\varepsilon \Delta_{1} + \bigo(\varepsilon^2)\right), \\ 
&= \left(I+\varepsilon \Delta_{\ell} \right) \left(I+\varepsilon
\Delta_{\ell-1}  \right) \cdots \left( I+\varepsilon \Delta_{1}
  \right) + \Xi.  
\end{align*} 
We now analyze the error term $\Xi$. First note
that for small enough $\varepsilon$ we have: 
\begin{equation*}
\left\| I+\varepsilon \Delta_{i} \right\| \le 1, \quad \text{for all } i = 1,
\ldots, m. 
\end{equation*}
Furthermore, since $[0,T]$ is fixed, the
number of temporal measurements $m$ is $m = T/\varepsilon = \bigo(1/\varepsilon)$.
Putting together these two facts, one obtains: 
\begin{equation*}
\Xi = \sum_{k=1}^{\ell} \binom{\ell}{k} \bigo(\varepsilon^{2k}) \le \sum_{k=1}^{\ell}
\frac{\ell^k}{k!} \bigo(\varepsilon^{2k}) =  \sum_{k=1}^{\ell} \frac{1}{k!}
\bigo(\varepsilon^k) = \bigo(\varepsilon),
\end{equation*}
which proves \eqref{eq:heat kernel expansion} for $t = \tau_{\ell}$. 

Now let $\tau_{\ell} > t$. The previous argument shows:
\begin{align*}
\T e^{\int_0^{\tau_{\ell}} \Delta_{g(\tau)} d\tau} &= \left(I+\varepsilon \Delta_{g(\tau_{\ell})} \right)
\left(I+\varepsilon \Delta_{\ell-1} \right) \cdots \left( I+\varepsilon
\Delta_{1} \right) + \bigo(\varepsilon), \\
\T e^{\int_0^t \Delta_{g(\tau)} d\tau} &= \left(I+\varepsilon \Delta_{g(t)} \right)
\left(I+\varepsilon \Delta_{\ell-1} \right) \cdots \left( I+\varepsilon
\Delta_{1} \right) + \bigo(\varepsilon).
\end{align*}
Therefore the difference between the two operators is:
\begin{equation*}
\| \T e^{\int_0^{\tau_{\ell}} \Delta_{g(\tau)} d\tau} - \T e^{\int_0^t
  \Delta_{g(\tau)} d\tau} \| \leq \varepsilon \|
\Delta_{g(\tau_{\ell})} - \Delta_{g(t)} \| + \bigo(\varepsilon) \leq C
\varepsilon^2 + \bigo(\varepsilon) = \bigo(\varepsilon),
\end{equation*}
and so \eqref{eq:heat kernel expansion} holds for any $0 < t \leq T$.

\end{proof}

\begin{proof}[Proof of Theorem \ref{thm:main}] Recall that each of our
diffusion operators $\PP_{\varepsilon,i}$ for a fixed $i$ is defined the
same as in the standard diffusion maps framework. Hence, by adapting
the convergence result from Singer \cite{Singer:2006}, we have that on
$E_K$:
\begin{equation*} \PP_{\varepsilon,i} = I + \varepsilon \Delta_i
+ \bigo\left(\frac{1}{n^{1/2} \varepsilon^{d/4-1/2}},\varepsilon^2
\right).  \end{equation*} 
Let $\ell = \lceil t/\varepsilon \rceil$ and recall our $\ell$-step transition
operator $\PP_{\varepsilon}^{(\ell)}$ is defined as $\PP_\varepsilon^{(\ell)}
\defeq  \PP_{\varepsilon,\ell} \PP_{\varepsilon,\ell-1} \cdots
\PP_{\varepsilon,1}.$ The proof of Lemma \ref{prop:approximation} demonstrates that
the error of the product increases by a factor of $\varepsilon$, hence:
\begin{equation*}
\PP^{(\ell)}_{\varepsilon} = \left(I+\varepsilon \Delta_{\ell} \right)
\left(I+\varepsilon \Delta_{\ell - 1} \right) \cdots \left( I+\varepsilon
\Delta_{1} \right) + \mathcal{O} \left(\frac{1}{n^{1/2} \varepsilon^{d/4+1/2}},\varepsilon \right).
\end{equation*}
Therefore using Lemma \ref{prop:approximation} we conclude that
\begin{equation*} 
\PP_{\varepsilon}^{(\lceil t/\varepsilon \rceil)} = \T e^{\int_0^t
\Delta_{g(\tau)} d\tau}  + \mathcal{O} \left(\frac{1}{n^{1/2}
\varepsilon^{d/4+1/2}},\varepsilon \right),
\end{equation*}
on the set $E_K$, where $K$ is fixed but arbitrary. The result follows by taking
the union and closure $\overline{\bigcup_{K > 0} E_K}$, as in \cite{Coifman:2006}.
\end{proof}

\begin{rmk*} The connection to heat diffusion can be also considered in terms of
infinitesimal generators.  For the static case, the product
$\PP_\epsilon^{t/\epsilon}$ approaches a continuous-time Markov operator $P_t$,
which can be represented in terms of a generator $G$ as $P_t = e^{t G}$ where
$G = \Delta$. In the above, we have shown that the $\lceil t/\varepsilon
\rceil$-step transition operator for our inhomogeneous Markov chain
$\PP^{(\lceil t/\varepsilon \rceil)}_\epsilon$ approaches  the continuous
process $\T e^{\int_0^t \Delta_{g(\tau)} d\tau}$, whose time-dependent
infinitesimal generator at time $t$ is $\Delta_{g(t)}$.  
\end{rmk*} 
\section{Numerical Example} \label{sec:numerical} Suppose that $D$ consists of
$n = 10,000$ points sampled uniformly at random from the unit disc in $\bR^2$.
Define the maps $h,v : \R^2 \rightarrow \R^2$ by 
\[ h(x,y) = \big(x,y \cdot(1-\cos \pi x)\big), \quad \text{and} \quad v(x,y) =
\big(x \cdot (1-\cos \pi y),y\big).  \]
We refer the images $h(D)$ and $v(D)$ as horizontal and vertical barbells,
respectively. In the following, we define a deformation of the $n = 10,000$
points from a horizontal to a vertical barbell over $m=9$ times.  We define $X_1
= h(D)$, $X_5 = D$, and $X_9 = v(D)$. The intermediate sets $X_2$, $X_3$, $X_4$
and $X_6$, $X_7$, $X_8$ are defined by linear interpolation. Specifically, if
$X_i = \left\{x_j^{(i)} \right\}_{j=1}^n$, then
$$
x_j^{(i)} = \left( x_j^{(5)} - x_j^{(1)}\right) \cdot \frac{i-1}{4} + x_j^{(1)}
\quad \text{and} \quad
x_j^{(i)} = \left( x_j^{(9)} - x_j^{(5)}\right) \cdot \frac{i-5}{4} +
x_j^{(5)},
$$
for $i = 2,3,4$ and $i=6,7,8$, respectively.  In the first row of Figure
\ref{fig:barbell} the data $X_i$ is plotted for $i=1,5,9$.
\begin{figure}[h!]  
\centering 
\begin{tabular}{c|c|c|c} & $i=1$
& $i=5$  & $i=9$ \\ & & & \\ \hline \rotatebox{90}{\quad \quad \,\,
Data} & \includegraphics[width=.29\textwidth]{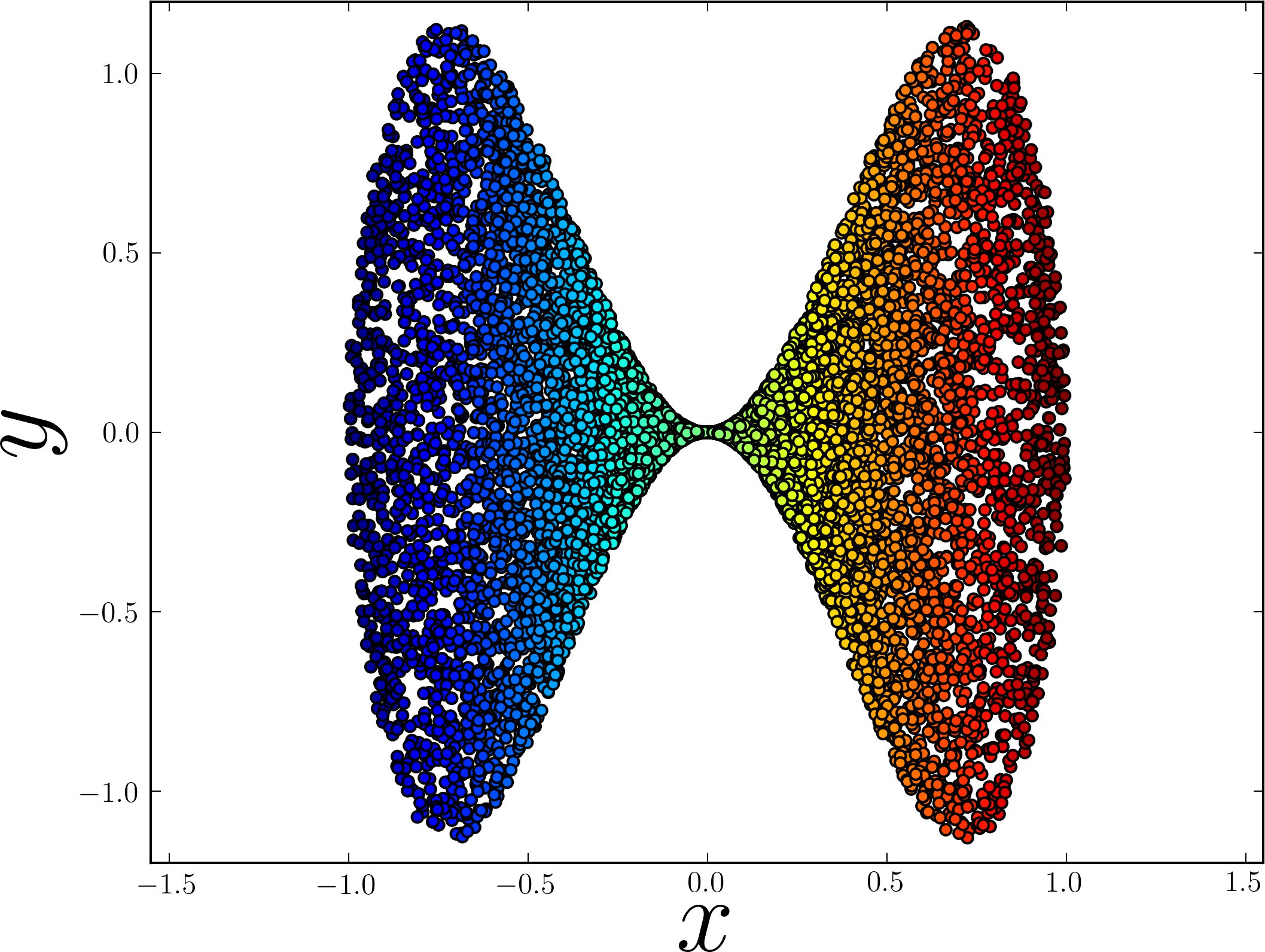} &
\includegraphics[width=.29\textwidth]{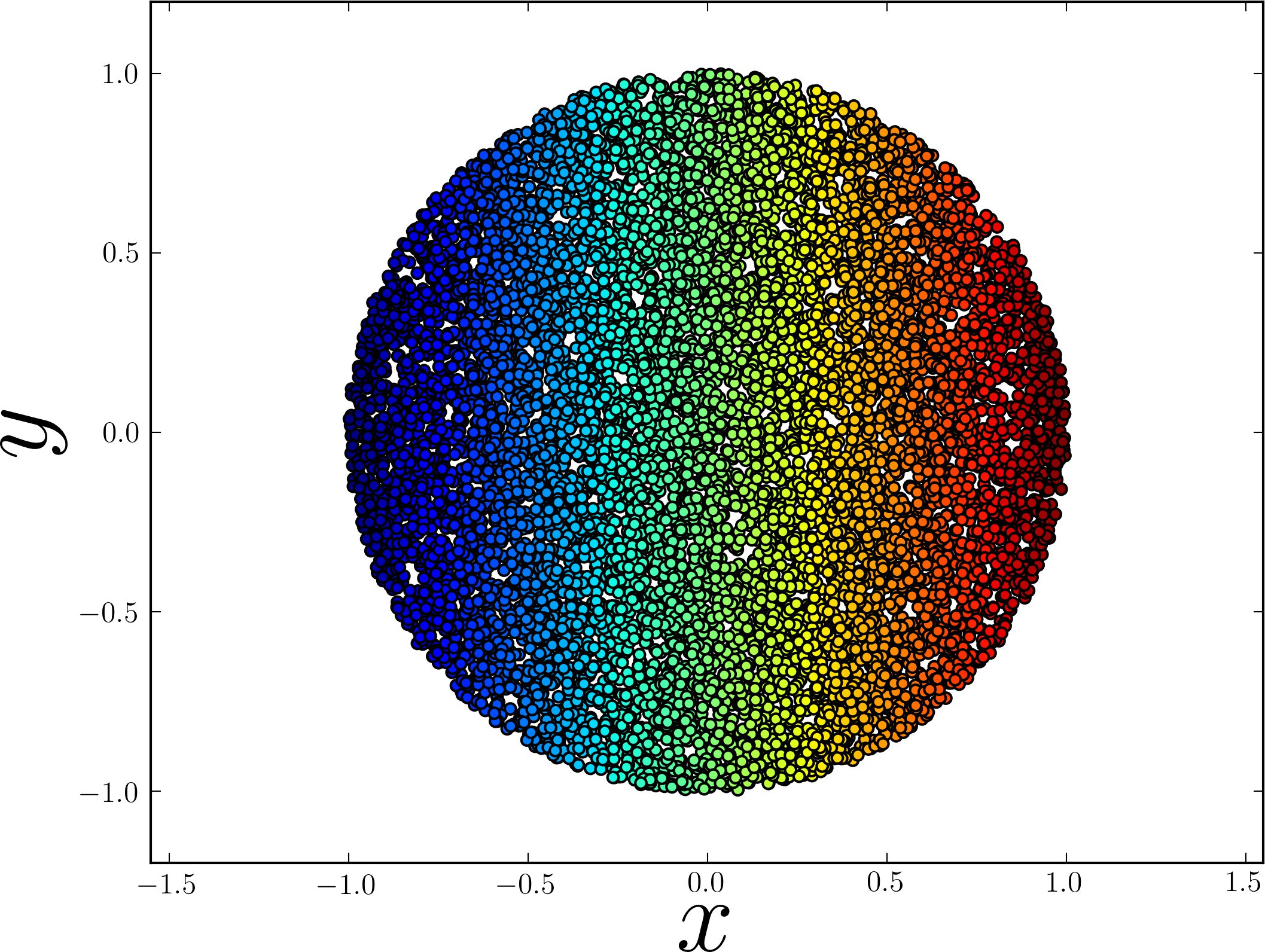} &
\includegraphics[width=.29\textwidth]{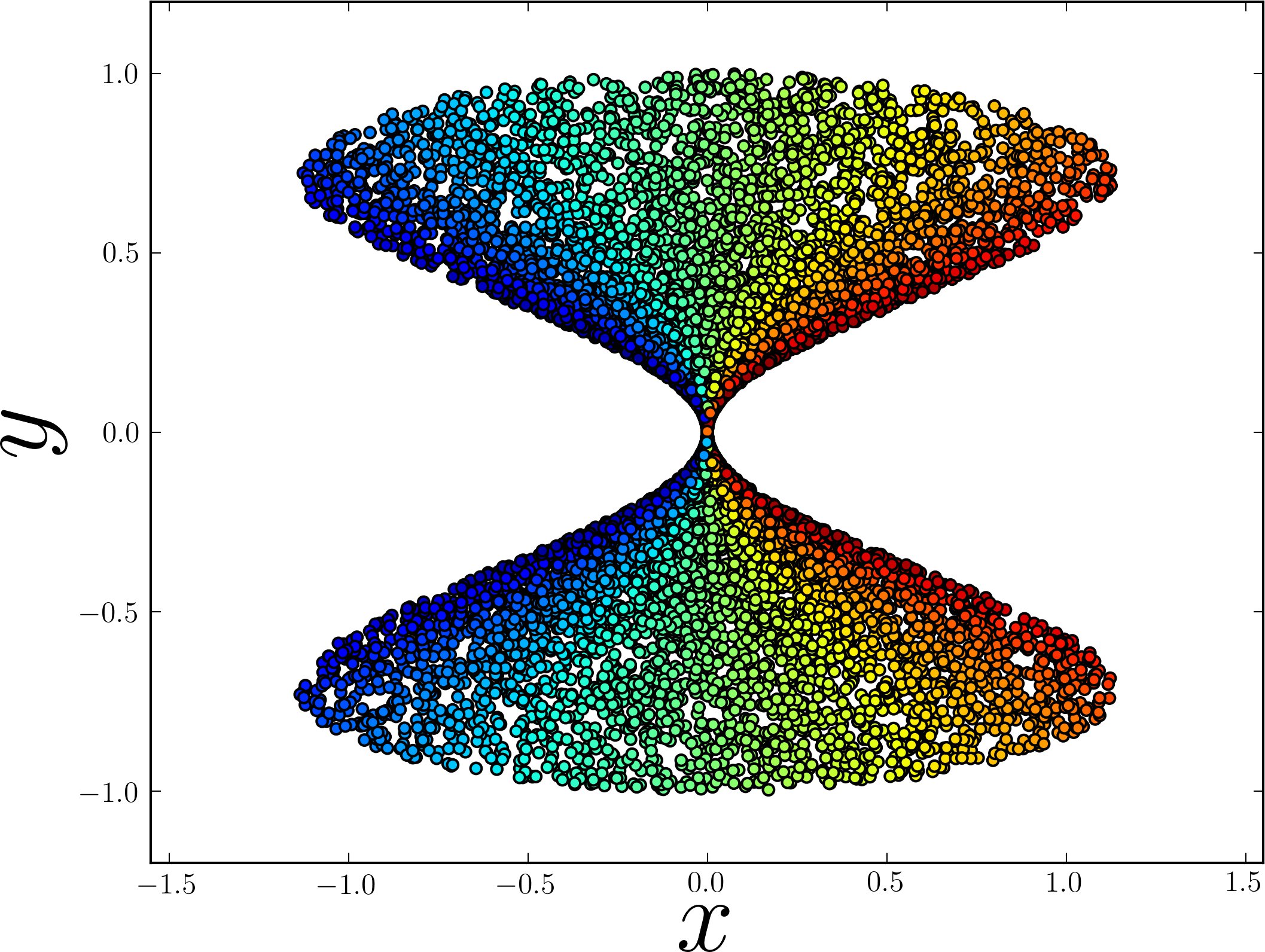} \\ \hline
\rotatebox{90}{\,\,\, Diffusion map} &
\includegraphics[width=.29\textwidth]{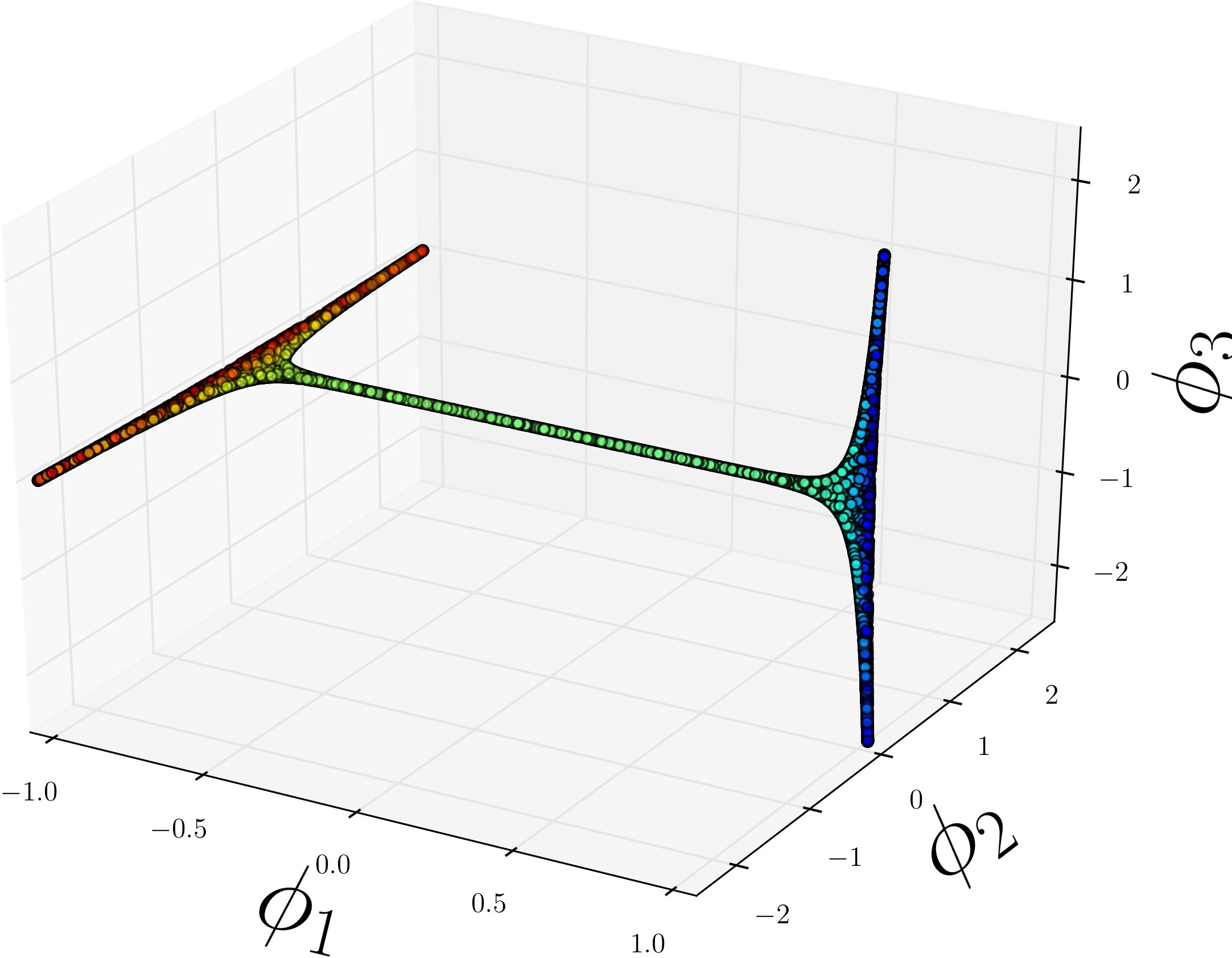} &
\includegraphics[width=.29\textwidth]{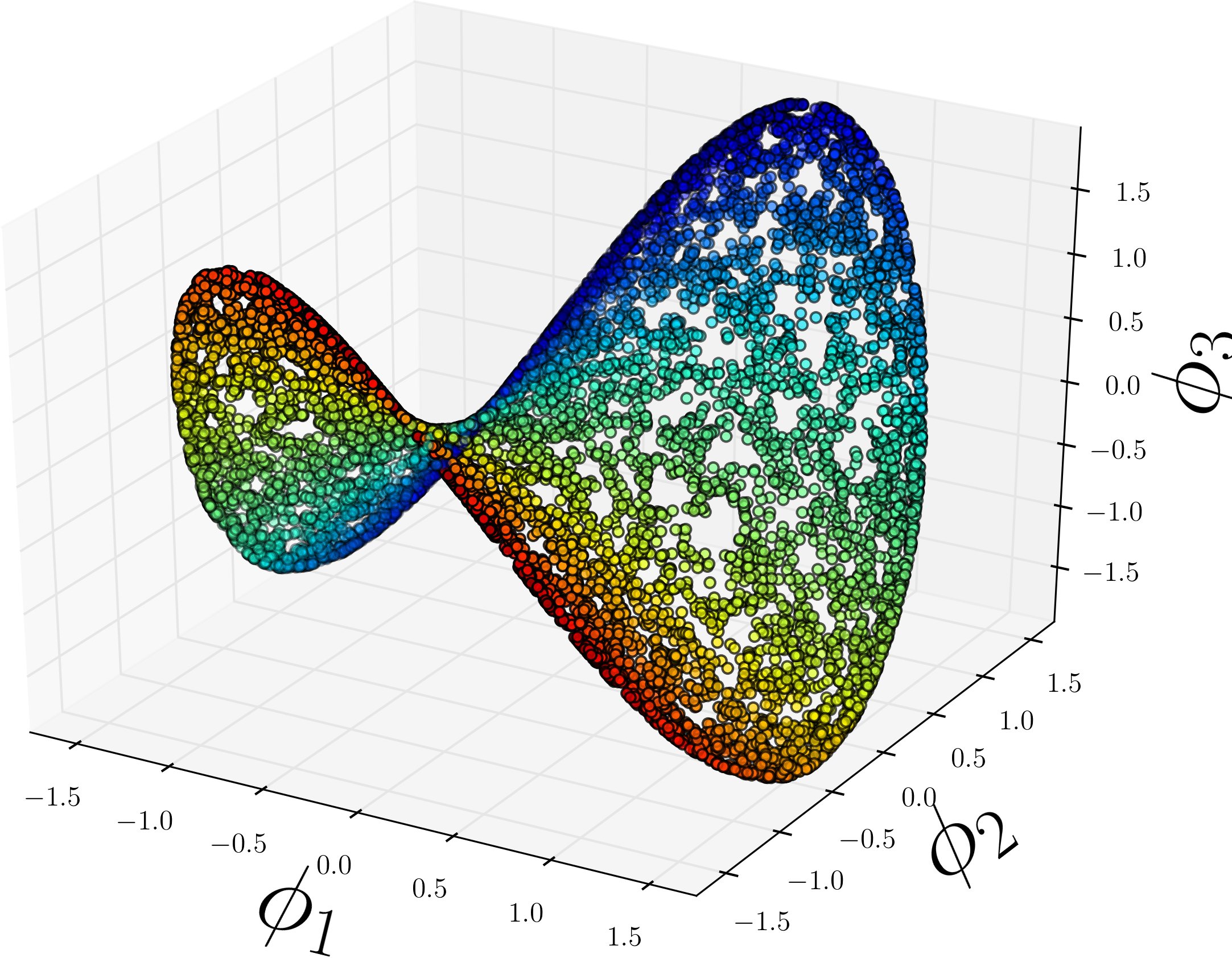} &
\includegraphics[width=.29\textwidth]{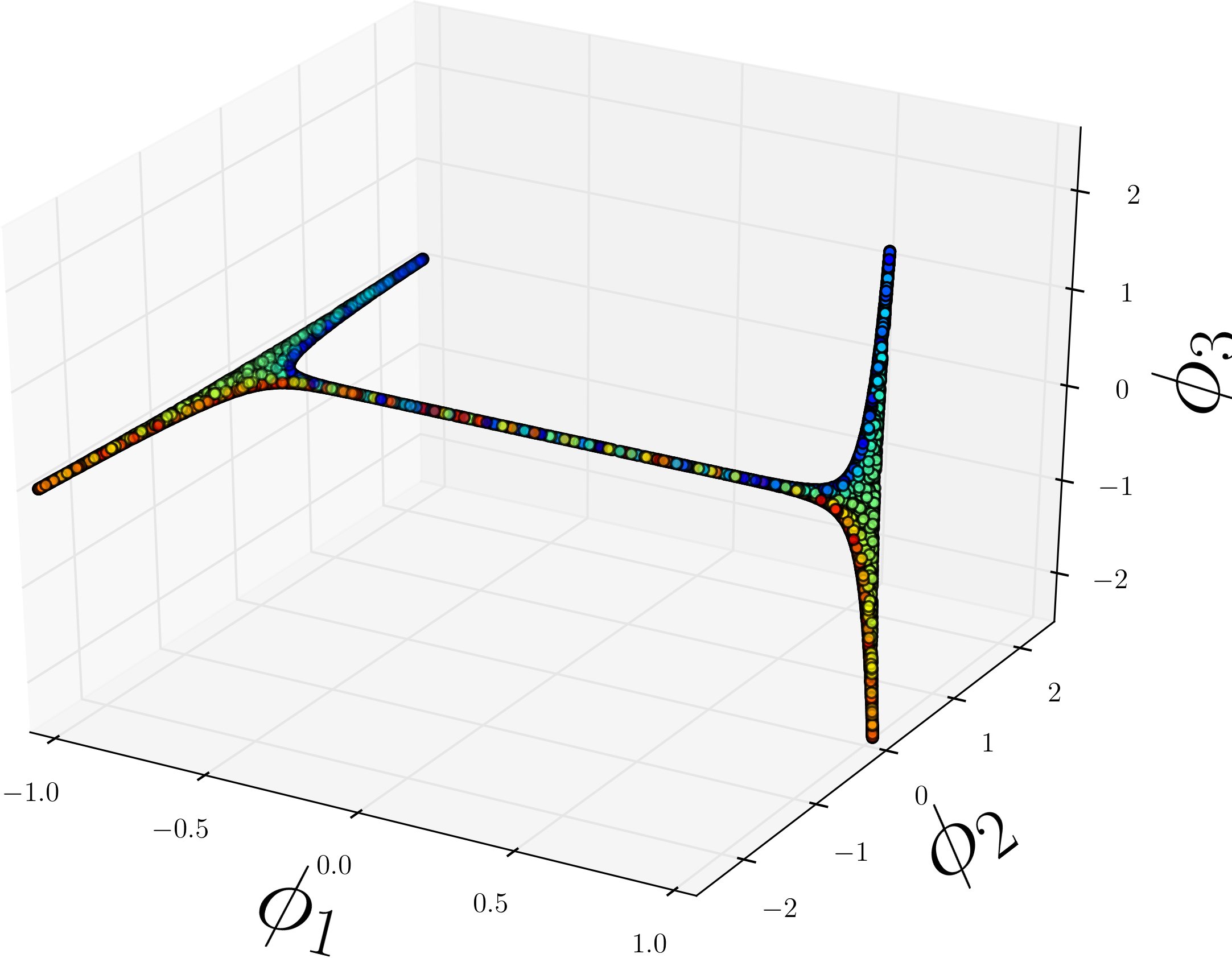} \\ \hline
\rotatebox{90}{\,Concatenated data DM\,\,} &
\includegraphics[width=.29\textwidth]{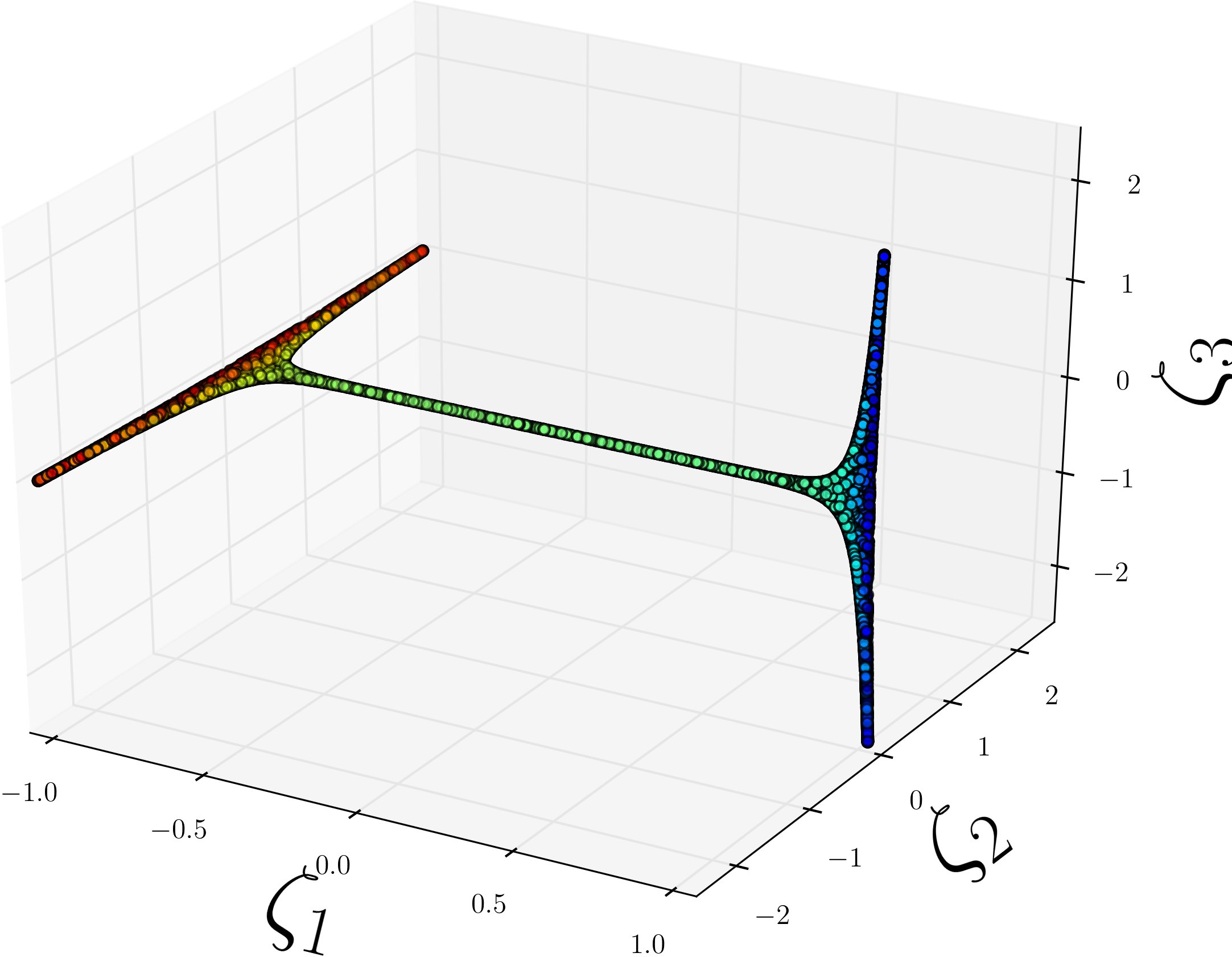} &
\includegraphics[width=.29\textwidth]{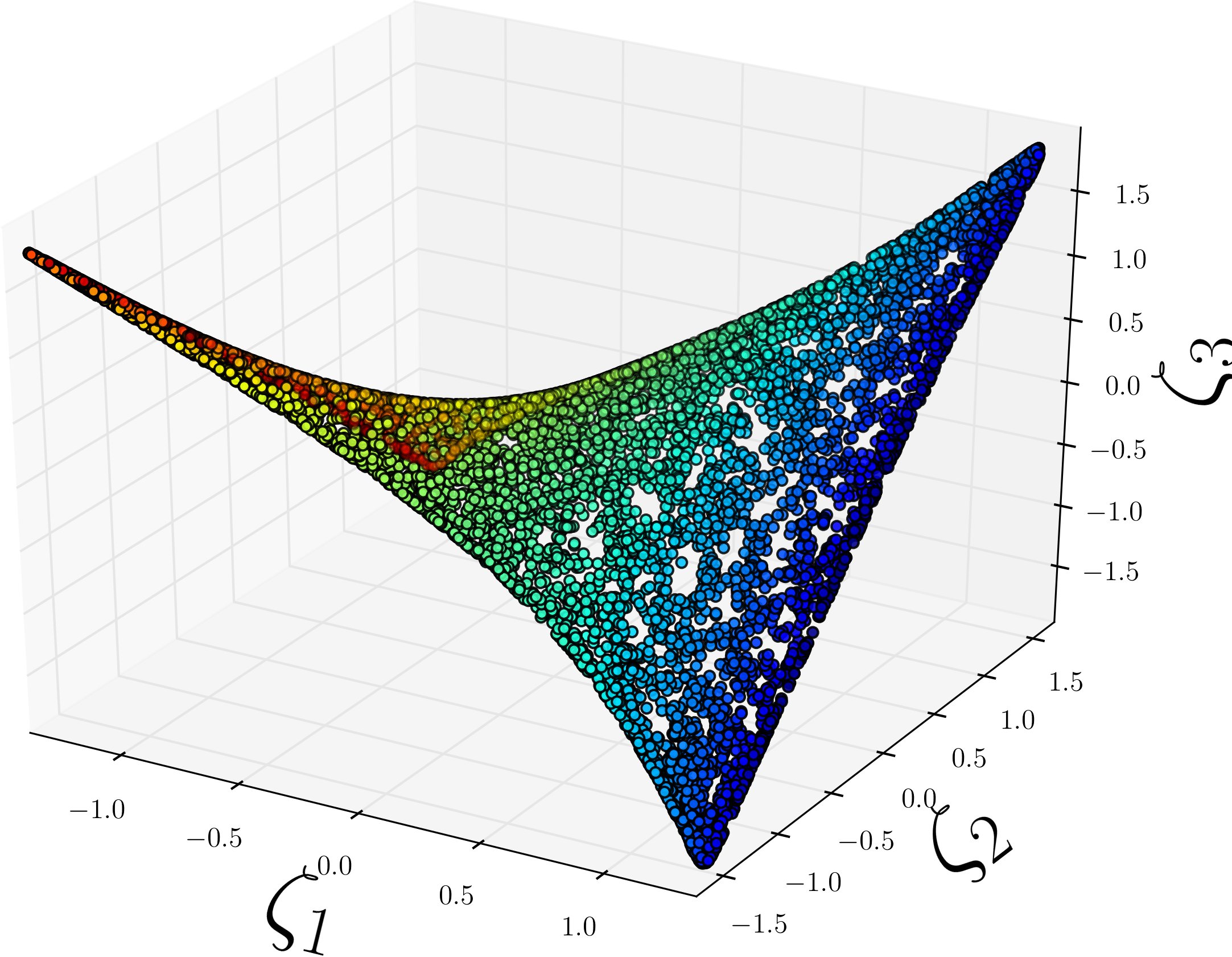} &
\includegraphics[width=.29\textwidth]{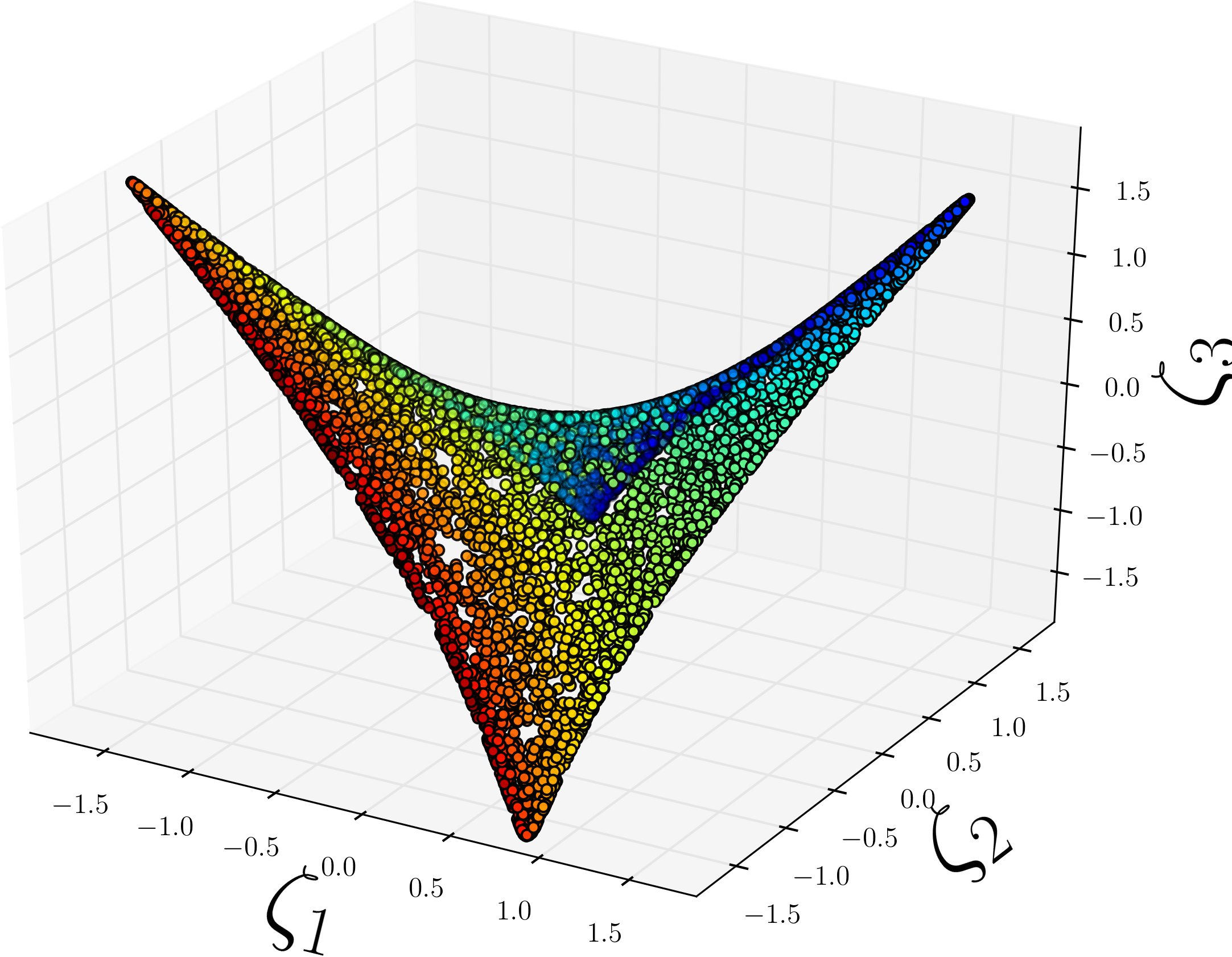} \\ \hline \rotatebox{90}{
Time coupled DM \,\,} & \includegraphics[width=.29\textwidth]{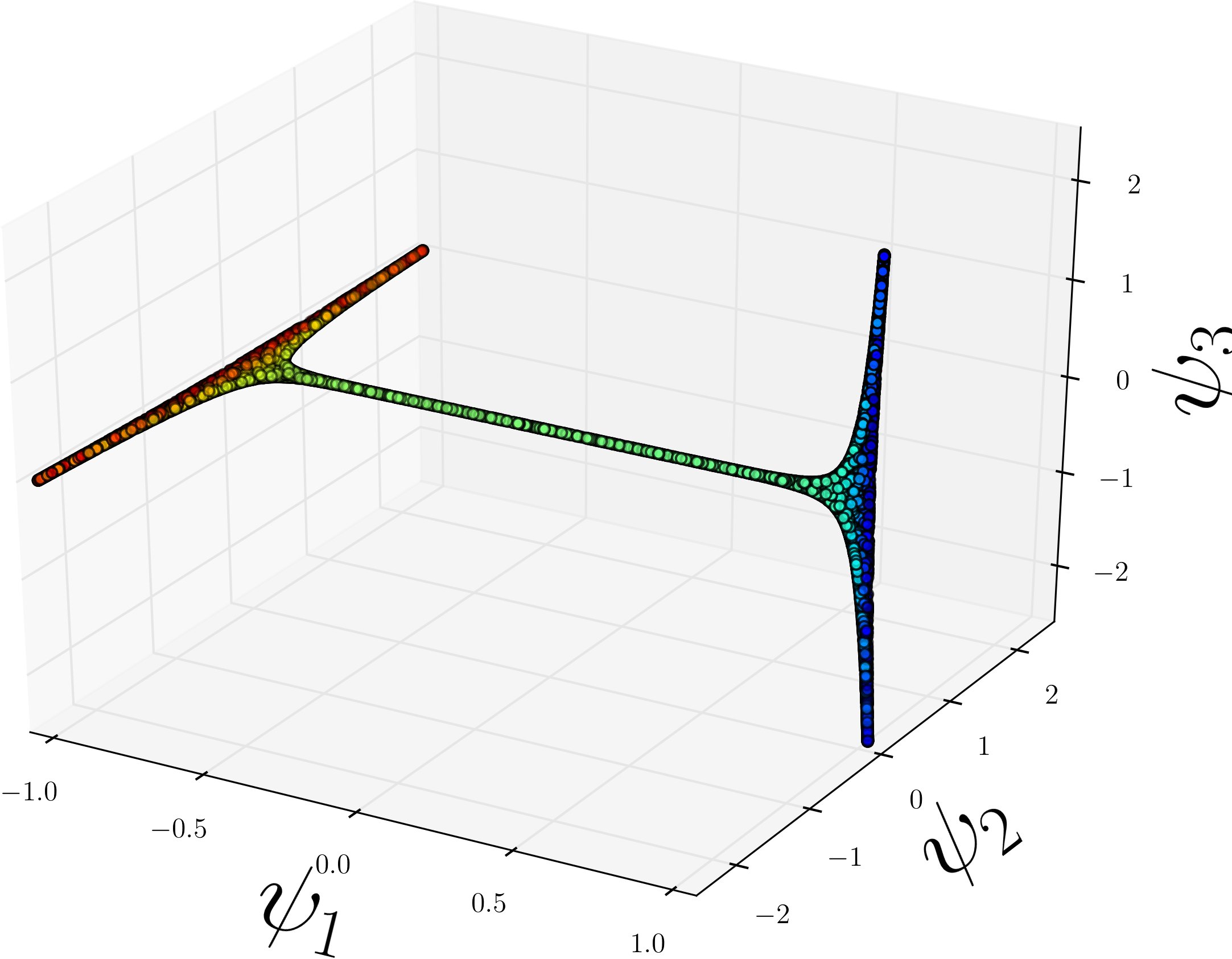} &
\includegraphics[width=.29\textwidth]{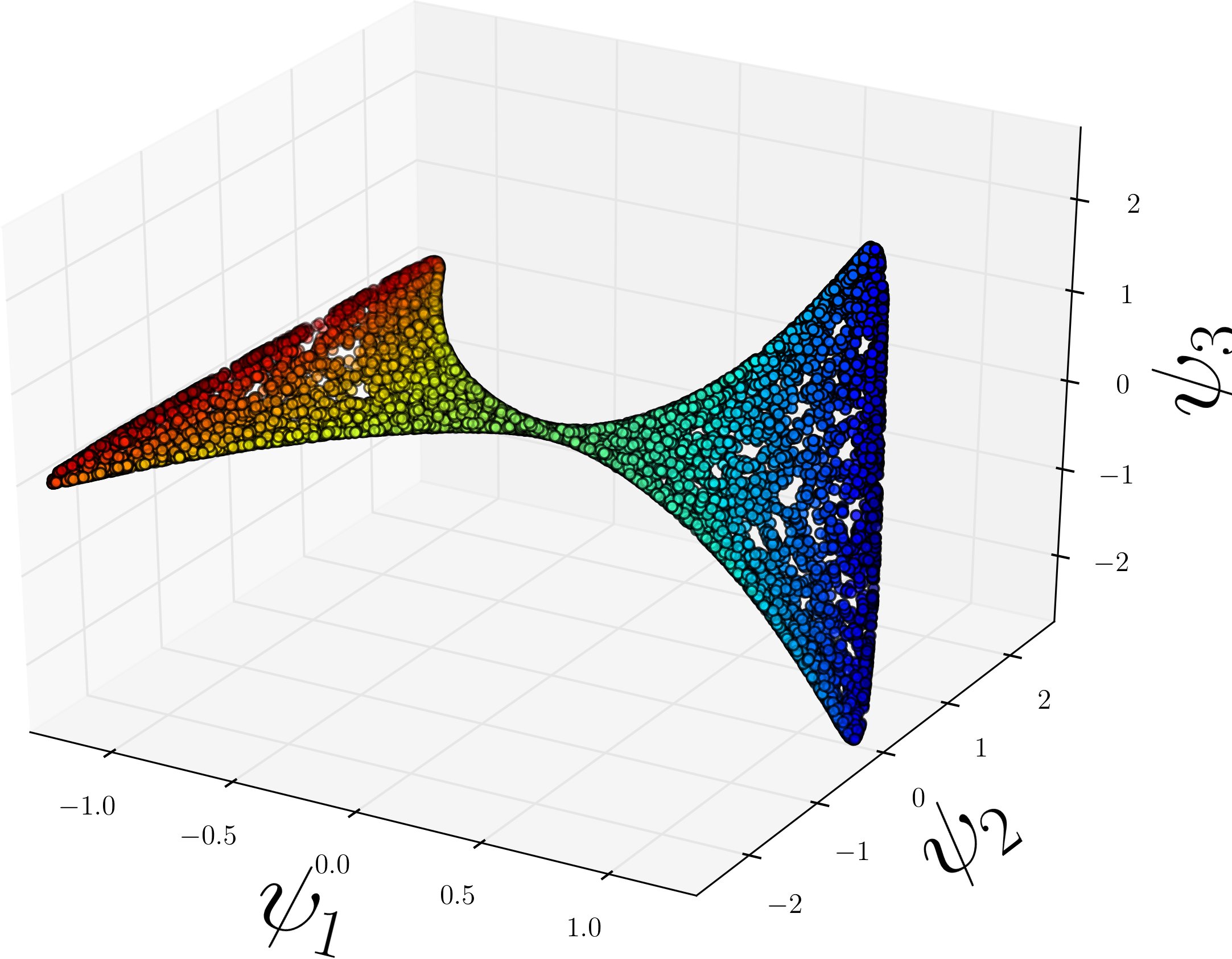} &
\includegraphics[width=.29\textwidth]{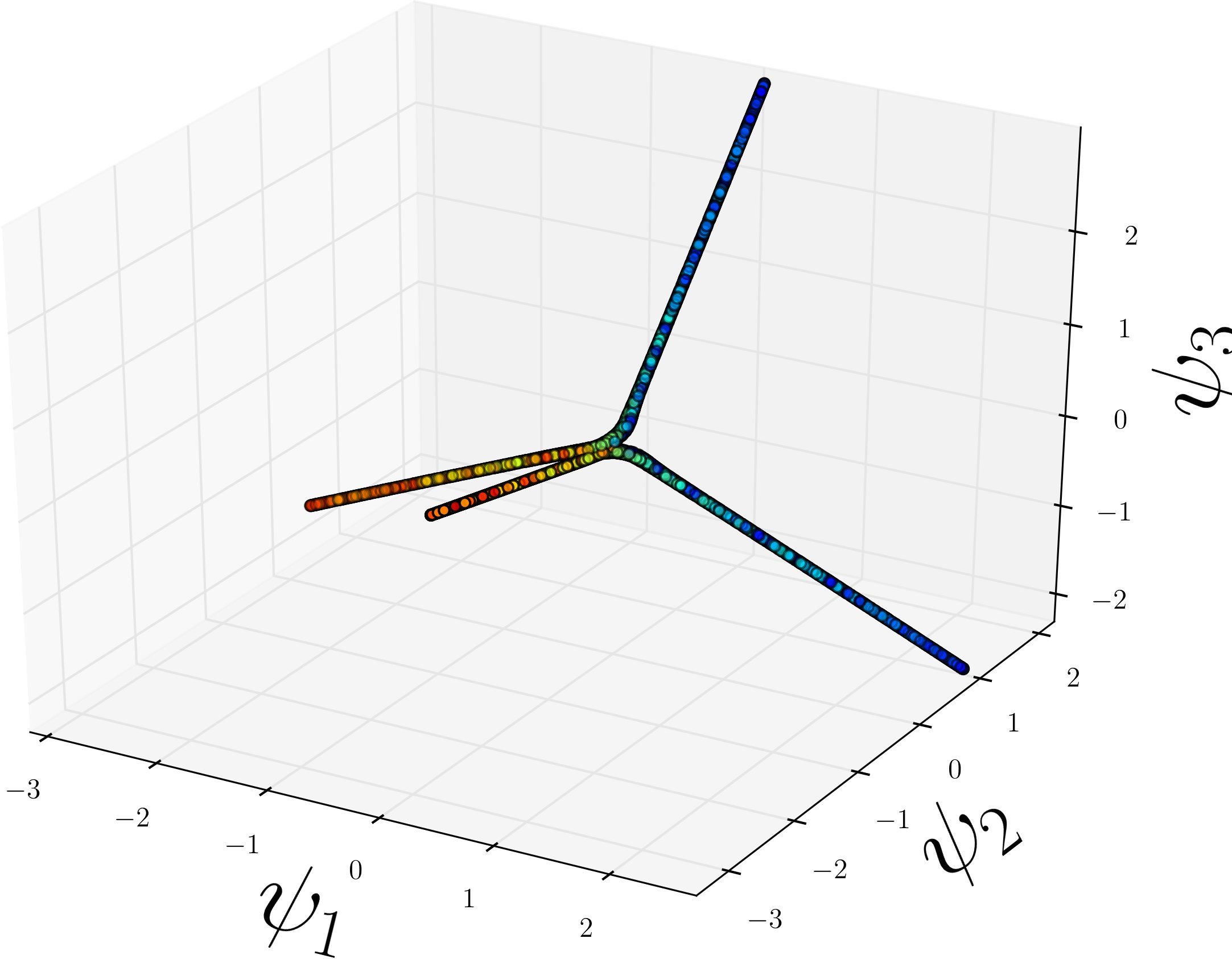} \\ 
\end{tabular}
\caption{In the first row, the data sets $X_i$ are plotted for times $i=1,5,9$;
in the second row, the first three coordinates of the diffusion map of the fixed
sets $X_1,X_5,$ and $X_9$ are plotted; in the third row, the first three
coordinates for the concatenated data diffusion map (as described in Section
\ref{sec:numerical}) for the sequences of data sets $\{X_i\}_{i=1}^1$,
$\{X_i\}_{i=1}^5$, and $\{X_i\}_{i=1}^9$ are plotted; finally, in the fourth
row, the time coupled diffusion map for the sequences of data sets
$\{X_i\}_{i=1}^1$, $\{X_i\}_{i=1}^5$, and $\{X_i\}_{i=1}^9$ are plotted. A
consistent color map is used across all plots.} \label{fig:barbell}
\end{figure}

From a computational point of view the sequence of data sets $(X_1,\ldots,X_9)$ can
be stored as an $10,000 \times 2 \times 9$ data tensor corresponding to the
$10,000$ points in $\bR^2$ measured over the $9$ times. Moreover, by fixing the
first coordinate of this tensor, we have a $2 \times 9$ matrix corresponding to
the trajectory of a single point through $\bR^2$ over the $9$ times.  Note how,
as discussed in Section \ref{sec:2.2}, even though points were sampled uniformly
at time $i=5$, the deformation of the data manifold creates a
nonuniform distribution of points at the other times. In the following, we
compare three methods of embedding the data.

First, as a baseline, we compute the diffusion map for each fixed data
set $X_1$, $X_5$, and $X_9$ and plot the first three embedding coordinates
$\phi_1,\phi_2,\phi_3$ in the second row of Figure \ref{fig:barbell}.  Second,
we define a concatenated data diffusion map based on the concatenation of the
data up to time $i$. That is, at time $i$ we consider the data
$$
X_{(i)} := \left\{ \left(x_j^{(1)},\ldots,x_j^{(i)} \right) \right\}_{j=1}^n
\subseteq \mathbb{R}^{2 i},
$$
and construct a diffusion map via equations  \eqref{eqn:markov1},
\eqref{eqn:markov2}, \eqref{eqn:markov3}, and \eqref{eqn:markov4}, where the
kernel $\K_{\varepsilon, i}(x_j,x_k)$  in equation \eqref{eqn:markov1} is
replaced by the Gaussian kernel $\K^\text{concat}_{\varepsilon, i}(x_j,x_k)$
defined by
$$
\K^\text{concat}_{\varepsilon, i}(x_j,x_k)  := 
\exp \left( - \frac{\left\| \left( x_j^{(1)},\ldots,x_j^{(i)} \right) - \left(
x_k^{(1)},\ldots,x_k^{(i)} \right) 
\right\|^2_{\ell^2 \left(\mathbb{R}^{2 i} \right)}}{\varepsilon} \right).
$$
The first three coordinates $\zeta_1,\zeta_2,\zeta_3$ of the concatenated data
diffusion map are plotted in the third row of Figure \ref{fig:barbell}.
Finally, for $\{X_i\}_{i=1}^1$, $\{X_i\}_{i=1}^5$, and $\{X_i\}_{i=1}^9$ we
computed the time coupled diffusion map and plot the first three coordinates
$\psi_1,\psi_2, \psi_3$ in the fourth row of Figure \ref{fig:barbell}. We
would like to draw the reader's attention to the difference between the
concatenated data diffusion map and the time coupled diffusion map at time
$i=9$. To facilitate visualization, we have plotted the first two coordinates of
these embeddings in Figure \ref{fig3}.
\begin{figure}[h!]
\centering
\begin{tabular}{cc}
\includegraphics[width=.35\textwidth]{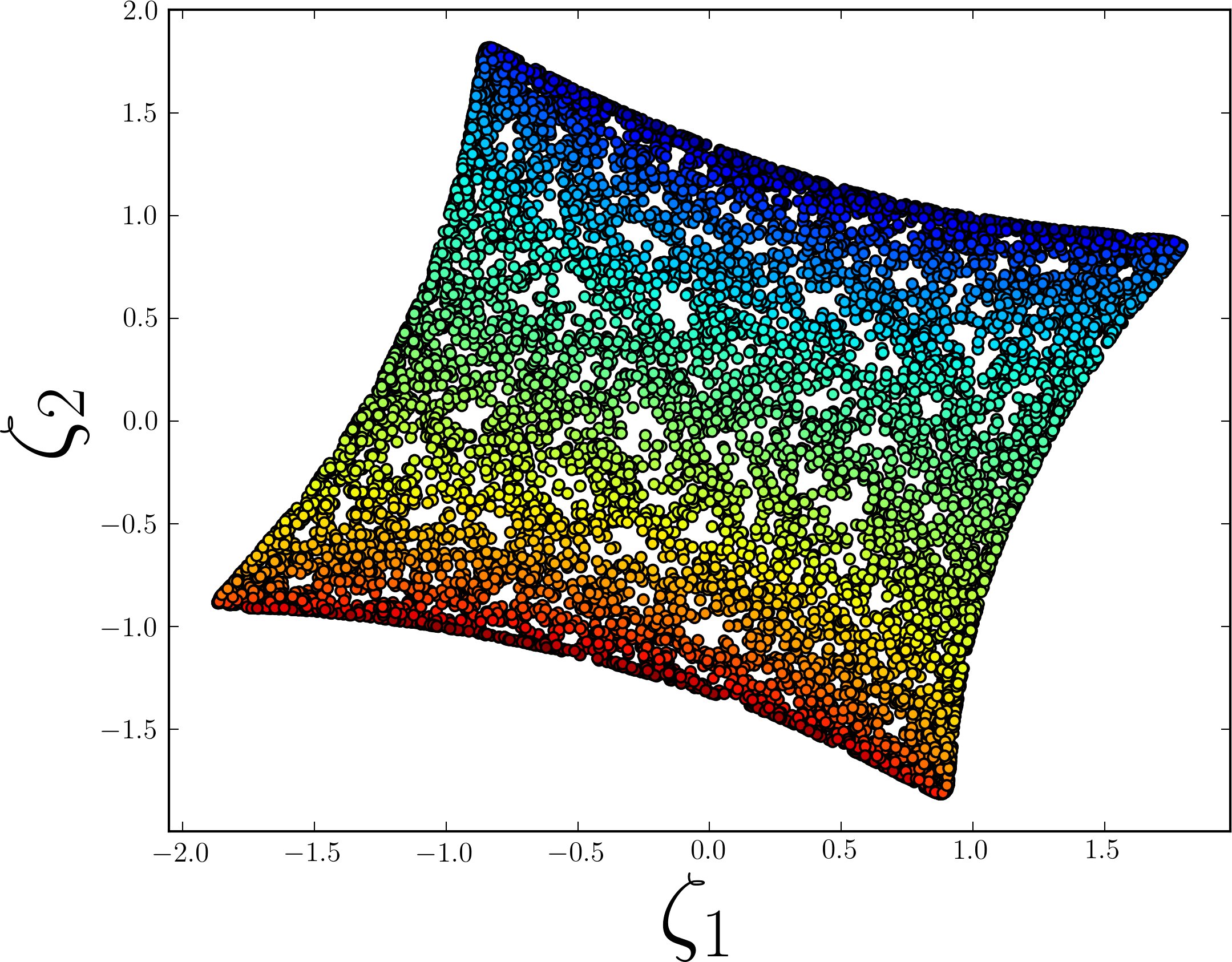} &
\includegraphics[width=.35\textwidth]{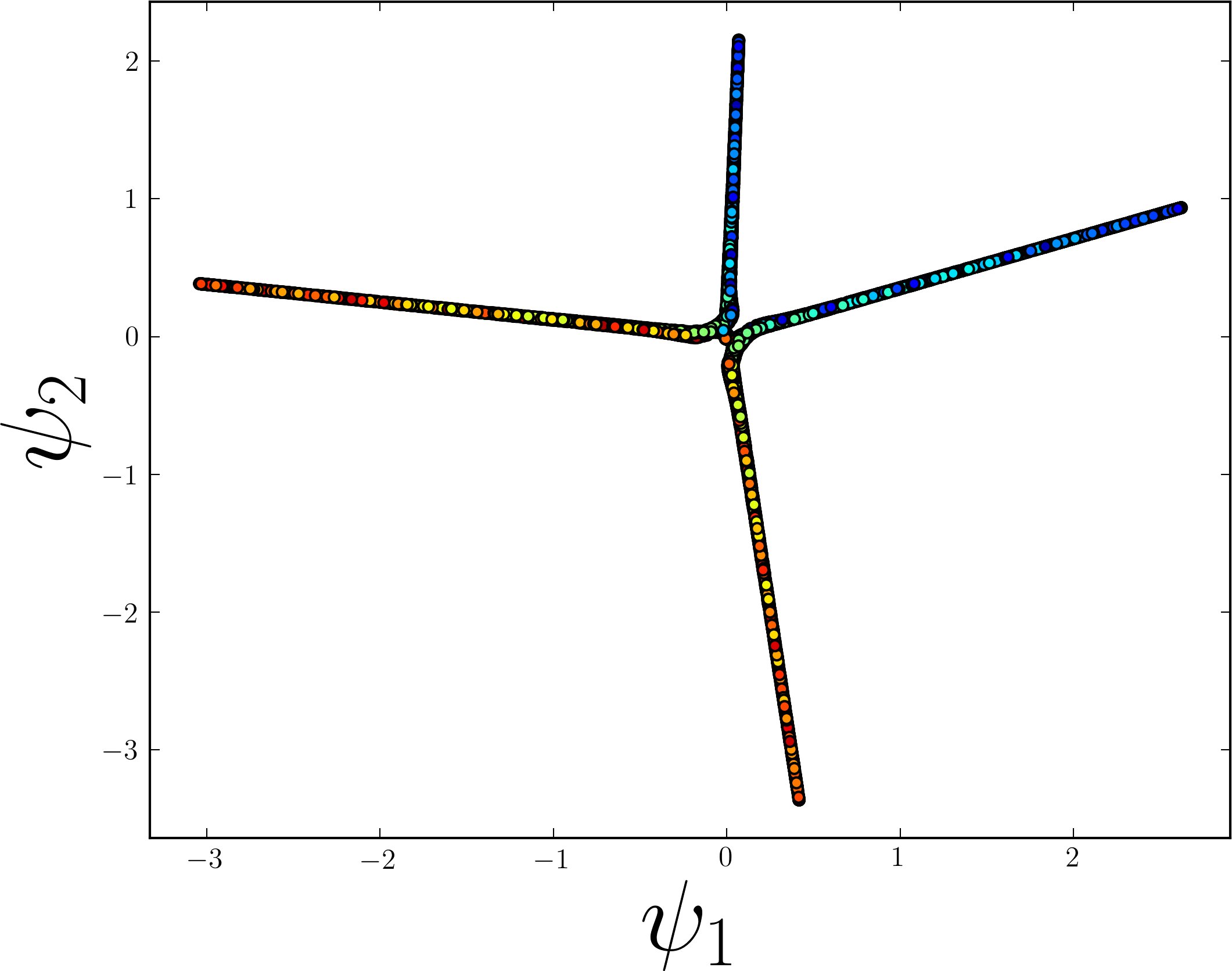}
\end{tabular}
\caption{For time $i=9$ the first two coordinates of the concatenated data
DM (left) and time coupled DM (right) are plotted.} \label{fig3}
\end{figure}

Observe that each point $x_j$ in the data $\left\{x_j^{(i)} \right\}_{j=1}^n$,
$i=1,\ldots,9$ can be assigned one for four classes: 
$$ 
\text{(right-up), (right-down), (left-up), or (left-down),}
$$
corresponding to side (left or right) that the point resides in the horizontal
barbell at $i=1$, and the side (up or down) that the point resides in the
vertical barbell at time $i=9$. Each of the four lines in the time coupled diffusion
map embedding correspond to one of these four classes. On the other hand, we
interpret the concatenated data diffusion map by observing that the kernel
$\K^\text{concat}_{\varepsilon, i}(x_j,x_k)$ can also be written
$$
\K^\text{concat}_{\varepsilon, i}(x_j,x_k)  = 
\exp \left( -\frac{\sum_{l=1}^i \|x_j^{(l)} -
x_k^{(l)}\|^2_{\ell^2 \left(\mathbb{R}^2 \right)}}{\varepsilon} \right) 
= \prod_{l=1}^i \exp \left( - \frac{\|x_j^{(l)} - x_k^{(l)}
\|^2_{\ell^2 \left(\mathbb{R}^2 \right)}}{\varepsilon} \right).
$$
Thus, concatenating the available data is equivalent to either averaging (up to
a constant) the square distances for the available times, or multiplying the
Gaussian kernels pointwise for the available times. The resulting concatenated
data diffusion map encodes the fact that on average the points are spread out
over the disc; the only indication of the existence of the barbells at times
$i=1$ and $i=9$ are the corners in the embedding. In contrast, the time coupled
diffusion map embedding clearly captures the two barbell events: the four lines
in this embedding represent the four classes in the
data induced by the barbell events at times $i=1$ and $i=9$. Moreover, this
numerical example demonstrates that the time coupled diffusion map embedding
aggregates information over time. Indeed, the two barbell events are
temporally disjoint and thus no single temporal slice of the data is
sufficient to recover the four classes in the data. This temporal disjointness
is reflected in the standard diffusion map embeddings in row two of Figure
\ref{fig:barbell}. In summary, the numerical results provide an example where
averaging square distances over time, or equivalently multiplying Gaussian
kernels pointwise over time, is insufficient to capture geometric events in the
evolution of the data; instead, by viewing the data as an evolving manifold and
approximating heat flow on the manifold via the product of diffusion operators,
we are able to define an embedding which provides a concise description of the
structures which occur in the evolution of the data.

\section{Continuous time coupled diffusion map} \label{sec: cont tcdm} 
Theorem \ref{thm:main} approximates the operator $T_Z^{(t)}: L^2(\M) \rightarrow
L^2(\M)$, defined in \eqref{eqn: TZt def}, with the discrete inhomogeneous
Markov chain $\PP_{\varepsilon}^{(\lceil t/\varepsilon \rceil)}$. The
inhomogeneous chain $\PP_{\varepsilon}^{(\lceil t/\varepsilon \rceil)}$ is the
foundation of the time coupled diffusion distance, and after the normalization
in \eqref{eqn:PiP} the singular value decomposition can be used to construct the
corresponding time coupled diffusion map \eqref{eqn:tcdm}. In this section we further investigate the connection between the heat kernel $Z$ and
the time coupled diffusion map through the continuous operator $T_Z^{(t)}$.

Recall the diffusion operator $T_Z^{(t)} f(x) = \int_{\M} Z(x,t; y,0) f(y) \,
dV(y,0)$ introduced in Section \ref{sec:manifold}. If $f$ is an initial
distribution over $\M$, the operator $T_Z^{(t)}$ diffuses $f$ forward in time
over the changing manifold geometry, up to time $\tau = t$. The adjoint of this
operator is $\left(T_Z^{(t)}\right)^{\ast}: L^2 (\M) \rightarrow L^2 (\M)$,
\begin{equation*}
\left(T_Z^{(t)}\right)^{\ast} f(x) = \int_{\M} Z(y,t; x,0) f(y) \, dV(y,t).
\end{equation*}
It follows from \eqref{eqn: row sums one} that the operator
$\left(T_Z^{(t)}\right)^{\ast}$ maps a probability distribution over $(\M, g(t))$
backwards in time to a probability distribution over $(\M, g(0))$.
Indeed, if $f$ is a probability density function, then
$\left(T_Z^{(t)}\right)^{\ast}f$ is its posterior probability distribution as in
Section \ref{sec:tcdm}. In the continuous setting, the analysis is performed
over the manifold $(\M,g(\cdot))$, and all integration is performed with respect
to the Riemannian volume on the manifold similar to the approach in
\cite{abdallah:embedTimeManifold2012,Abdallah:2010,
berard:embedManifoldHeatKer1994}.  In the continuous case, the stationary
distribution is constant and the normalization in \eqref{eqn:PiP} is therefore
unnecessary. Moreover, the operator $T_Z^{(t)}$ is bi-stochastic in the sense
that the constant function is an eigenfunction of both $T_Z^{(t)}$ and $\left(
T_Z^{(t)} \right)^*$. Indeed, in Lemma \ref{prop: heat kernel 1st order
expansion} 
we showed that $T_Z^{(t)}$ can be written as the limit of a
product -- with an increasing number of terms -- of self-adjoint operators, which each
have the constant eigenfunction. Therefore, the continuous analog of the
discrete diffusion distance \eqref{eqn: tcdd} is the diffusion distance 
\begin{equation*}
D^{(t)} (x,y)^2 = \int_{\M} ( Z(x,t; w, 0) - Z(y,t; w, 0))^2 \, dV(w,0),
\end{equation*}
or equivalently, 
\begin{equation*}
D^{(t)} (x,y) = \| \left(T_Z^{(t)}\right)^{\ast} \delta_x -
\left(T_Z^{(t)}\right)^{\ast} \delta_y \|_{L^2(\M)}, 
\end{equation*}
where $\delta_x$ is the Dirac
distribution centered at $x \in \M$ (compare to
\eqref{eqn: tcdd}).  Since the ``row'' sums of $Z(x,t,y,0)$ are $1$ (see 
\eqref{eqn: row sums one}), and the manifold $\cM$ is compact:
\[
\int_{\M} \int_{\M} Z(x,t; y,0) dV(y,0) dV(x,t) < \infty
\]
Therefore, the operator $T_Z^{(t)}f(x) = \int_{\M} Z(x,t; y,0) f(y) \, dV(y,0)$
is compact, see for example \cite{Pedersen:2000}. Similarly, the operator
$\left(T_Z^{(t)}\right)^{\ast}$ is compact, and thus $T_Z^{(t)}
\left(T_Z^{(t)}\right)^{\ast}$ is compact, since it is the composition of two
compact operators.  Moreover, since $T_Z^{(t)} \left(T_Z^{(t)}\right)^{\ast}$ is
also self-adjoint, by the spectral theorem $T_Z^{(t)}
\left(T_Z^{(t)}\right)^{\ast}$ has an eigendecomposition, and similarly,
$\left(T_Z^{(t)}\right)^{\ast} T_Z^{(t)}$  has an eigendecomposition.
Therefore, analogous to the discrete decomposition \eqref{eqn:Asvd}, $T_Z^{(t)}$
has a singular value decomposition.
\begin{equation*} T_Z^{(t)}f(x) = \sum_{k\ge0} \sigma_k^{(t)} \langle f,
\varphi_k^{(t)} \rangle \psi_k^{(t)}(x), \end{equation*}
where $\varphi_k^{(t)} \in L^2(\M)$ is a right singular function of $T_Z^{(t)}$
(an eigenfunction of $\left(T_Z^{(t)}\right)^{\ast} T_Z^{(t)}$),
$\psi_k^{(t)} \in L^2(\M)$ is a left singular function of $T_Z^{(t)}$
(an eigenfunction of $T_Z^{(t)} \left(T_Z^{(t)}\right)^{\ast}$), and
$\{ \sigma_k^{(t)} \}_{k\ge0}$ are the corresponding singular values (the square
root of the eigenvalues of $\left(T_Z^{(t)}\right)^{\ast} T_Z^{(t)}$, or
equivalently the square root of the eigenvalues of $T_Z^{(t)}
\left(T_Z^{(t)}\right)^{\ast}$ since these operators share the same spectrum).
However, recall that in the discrete case, we must take into account the
stationary distribution of the Markov chain using the normalization in
\eqref{eqn:PiP}, since the discrete operator $\PP_{\varepsilon}^{(t)}$ is not in
general bi-stochastic. In the continuous setting left singular functions are
used to define the time coupled diffusion map \eqref{eqn:tcdm}, written here as:
\begin{equation*} 
\Psi^{(t)} (x) = \left( \sigma_k^{(t)} \psi_k^{(t)} (x) \right)_{k\ge1}.  
\end{equation*} 
where the index $k$ starts at $1$ because the $0$th left singular function is
constant.  

We remark that  the operator $T_Z^{(t)} \left(T_Z^{(t)} \right)^*$ can be
written as an integral operator
\begin{equation} \label{eqn: left (t) operator} 
T_Z^{(t)} \left( T_Z^{(t)} \right)^{\ast} f(x) = \int_{\M} K_Z^b (x,y;t) f(y) \, dV(y, t), 
\end{equation}
where
\begin{equation*} 
K_Z^b (x,y; t) = \int_{\M} Z(x, t; w, 0) Z(y, t; w, 0) \, dV(w,0).  
\end{equation*}
We refer to $K_Z^b(x,y;t)$ as the backwards kernel since it is integrated
against functions on $(M,g(t))$, i.e., integrated against functions with respect
to the Riemannian volume at the end of the time interval $[0,t]$.
Similarly, the operator $\left(T_Z^{(t)}\right)^\ast T_Z^{(t)}$ has integral
representation
\begin{equation*} 
\left(T_Z^{(t)} \right)^\ast T_Z^{(t)} f(x) = \int_{\M} K_Z^f (x,y;t) f(y) \, dV(y, 0), 
\end{equation*}
where
\begin{equation*} K_Z^f (x, y; t) = \int_{\M} Z(w, t; x, 0) Z(w, t; y, 0) \,
dV(w, t).  
\end{equation*}
We refer to $K_Z^f$ as the forward kernel since it is integrated against
functions on $(\M, g(0))$. The kernels $K_Z^b$ and $K_Z^f$ can be
interpreted as analogous to the matrices $\vA^{(t)} (\vA^{(t)})^\top$ and
$(\vA^{(t)})^\top \vA^{(t)}$ from the discrete case, whose eigenfunctions are
the left and right singular vectors of $\vA^{(t)}$, respectively. 

Since the time coupled diffusion map is based on the SVD of
$T_Z^{(t)}$ diffusing $\varphi_k^{(t)}$ forward in
time results in $\sigma_k^{(t)} \psi_k^{(t)}$, and conversely the backward
propagation of $\psi_k^{(t)}$ leads to $\sigma_k^{(t)} \varphi_k^{(t)}$; specifically
$T_Z^{(t)} \varphi_k^{(t)} = \sigma_k^{(t)} \psi_k^{(t)}$ and
$\left(T_Z^{(t)}\right)^{\ast} \psi_k^{(t)} = \sigma_k^{(t)} \varphi_k^{(t)}$. If we
define,
\begin{equation*}
\widetilde{\Phi}^{(t)}(x) = \left( \varphi_k^{(t)}(x) \right)_{k\ge1}
\end{equation*}
and set $T_Z^{(t)} \widetilde{\Phi}^{(t)} (x) = ( T_Z^{(t)} \varphi_k^{(t)} (x)
)_{k\ge1}$, then this observation leads to:

\begin{equation*}
D^{(t)} (x,y) = \| \left(T_Z^{(t)} \right)^{\ast} \delta_x - \left(T_Z^{(t)}
\right)^{\ast} \delta_y
\|_{L^2(\M)} = \| \Psi^{(t)}(x) - \Psi^{(t)}(y) \|_{\ell^2} = \|
T_Z^{(t)} \widetilde{\Phi}^{(t)} (x) - T_Z^{(t)} \widetilde{\Phi}^{(t)} (y) \|_{\ell^2}.
\end{equation*}

We see that the difference in posterior distributions of $\delta_x$
and $\delta_y$ is equivalent to the difference between the diffused
distributions of $\widetilde{\Phi}^{(t)}(x)$ and
$\widetilde{\Phi}^{(t)}(y)$; these equivalent distances are, in turn,
equal to the Euclidean distance between points under the time coupled
diffusion map $\Psi^{(t)}$.

\section{Conclusion} \label{sec: conclusion}

We have introduced the notion of time coupled diffusion maps as a method of
summarizing evolving data via an embedding into Euclidean space. In particular,
we have introduced a method of modeling evolving data as samples from a manifold
with a time-dependent metric. We show that the constructed time
inhomogeneous Markov chain  approximates heat diffusion over a manifold $(\M,
g(\cdot))$ with a smoothly varying family of metrics $g(\tau)$. In the context
of manifold learning, these operators and resulting embeddings are related to
the heat kernel of $\partial_t u = \Delta_{g(t)} u$, which provides geometric
and probabilistic interpretations. Numerical experiments indicate that the map
encodes aggregate geometrical information over arbitrarily long time scales,
and thus summarizes the geometry of a sequence of datasets in a useful way.

These results open new directions related to diffusion based manifold
learning, and in particular the use of inhomogeneous Markov chains and
asymmetric diffusion semi-groups to
understand data geometry. Early results on single cell data
\cite{welp:condensationSingleCell2016} show the usefulness of this
type of diffusion process in biology, but further numerical and
theoretical study is needed. Understanding precisely the nature of the
geometrical information encoded by the time coupled diffusion map
(extending \cite{Abdallah:2010, abdallah:embedTimeManifold2012}),
would give theoretical insight and could lead to further
developments. More immediately, the results contained here provide a
theoretical foundation for understanding dynamic data that can be
modeled as a time varying manifold. 

\section{Acknowledgements} 
N.M. was a participant in the 2013 Research Experience for
Undergraduates (REU) at Cornell University under the supervision of
M.H. During the REU program both were supported by the National
Science Foundation grant number NSF-1156350. This paper is the result
of work started during the REU. M.H. was supported by the European
Research Council (ERC) grant InvariantClass 320959 while writing the first
version of the paper. He is currently supported by the Alfred P. Sloan
Fellowship (grant number FG-2016-6607), the DARPA Young Faculty Award
(grant number D16AP00117), and the NSF (grant number 1620216).

Both authors would like to thank Ronald Coifman for numerous
insightful conversations, and the reviewers for their comments which improved
the manuscript.

\appendix

\section{Diffusion geometry from Markov kernels} \label{dmfrommarkov}

We show that the operator $\vA^{(t)}$, defined in Section
\ref{sec:tcdm}, line \eqref{eqn:PiP}, has
operator norm one and that $\vpi_{(t)}^{1/2}$, the square root of its
stationary distribution, is an eigenvector with eigenvalue one. We do
so by following \cite{Lafon:2004} and developing a diffusion geometry framework
starting from a Markov kernel. 

Suppose that $(\cX,dx)$ is a measure space equipped with a Markov
kernel $p(x,y)$. Moreover, we assume the kernel $p(x,y)$ nonnegative and has the
conservation property
\[ \int_\cX p(x,y) \, dy =1 \quad \text{for all } x \in \cX. \]
The kernel $p(x,y)$ induces a Markov operator $P^*$ from $L^2(\cX)$ to
itself defined by
\[ (P^*f)(x) = \int_\cX p(y,x) f(y) \, dy \quad \text{for all } x\in \cX.
\]
The notation $P^*$ has been used to bring attention to the fact that
$P^*$ is the adjoint of the operator $P$ defined as:
\begin{equation*}
(P f)(x) = \int_{\cX} p(x,y) f(y) \, dy.
\end{equation*}
We assume that $P^*$ has a unique strictly positive stationary distribution
$v^2(x)$, that is to say,
\[ (P^* v^2)(x) = \int_\cX p(y,x) v^2(y) \, dy = v^2(x) \quad \text{for all
} x \in \cX. \]
For example, if $\cX$ is a finite set with the counting measure (as in
Section \ref{sec:tcdm}), it
would suffice to assume that the Markov chain is irreducible and positive
recurrent. By the function $v(x)$, we denote the positive pointwise square root
of $v^2(x)$.
We define an operator $A : L^2(\cX) \rightarrow L^2(\cX)$ by 
\begin{equation} \label{defa}
 (A f)(x) =
\int_\cX \frac{v(x) p(x,y) }{v(y)} f(y) \, dy = \int_\cX a(x,y) d(y)
\, dy \quad \text{for all } x \in \cX.
\end{equation}
Here the kernel $a(x,y) = v(x) p (x,y)/v(y)$.
We refer to $A$ has an averaging or diffusion operator and show
it has several nice properties.

\begin{lem}
The operator norm $\|A\| = 1$, and the norm is achieved by the
function $v(x)$. 
\end{lem}

\begin{rmk*}
Taking $A = \vA^{(t)}$ establishes the claim from the beginning
of the appendix, which was originally stated after
line \eqref{eqn:PiP} in Section \ref{sec:tcdm}.
\end{rmk*}

\begin{proof}
First we evaluate $(Av)(x)$ to check the second part of the assertion
and establish a lower bound on the operator norm:
\[ (Av)(x) = \int_\cX v(x) p(x,y) \frac{v(y)}{v(y)} \, dy = v(x) \int_\cX
p(x,y) \, dy = v(x), \]
using the Markov property of the kernel. To establish that $1$ is an
upper bound on the operator norm we show that $\langle Af , Af \rangle
\le \|f\|^2$. First, observe that by Cauchy-Schwartz,
\begin{align}
\left| \int_\cX \frac{v^2(x) p(x,y) f(y)}{v(y)} \, dy \right| &\le
\left( \int_\cX v^2(x) p(x,y) \, dy \right)^{1/2}
\left( \int_\cX \frac{ v^2(x) p(x,y)}{ v^2(y)} f^2(y) \, dy \right)^{1/2},
\nonumber \\
&= v(x) \cdot \left( \int_\cX \frac{ v^2(x) p(x,y)}{ v^2(y)} f^2(y)
\, dy \right)^{1/2}. \label{cauchy1}
\end{align}
Now starting with $\langle Af, Af \rangle$ we multiply and divide by $v^2(x)$
yielding:
\[ \langle Af , A f \rangle = 
\int_\cX \frac{1}{v^2(x)} \int_\cX \frac{v(x)^2 p(x,y) f(y)}{v(y)} \, dy
\int_\cX \frac{v(x)^2 p(x,z) f(z)}{v(z)} \, dz \, dx. \]
By applying the inequality \eqref{cauchy1} we see that,
\[ \langle Af, Af \rangle \le \int_\cX  \left( \int_\cX \frac{ v^2(x)
p(x,y)}{v^2(y)} f^2(y) \, dy \right)^{1/2}
\left( \int_\cX \frac{ v^2(x)
p(x,z)}{v^2(z)} f^2(z) \, dy \right)^{1/2} \, dx. \]
Applying Cauchy-Schwartz again,
\[ \langle Af, Af \rangle \le \left(\int_\cX \int_\cX \frac{ v^2(x)
p(x,y)}{v^2(y)} f^2(y) \, dy \, dx \right)^{1/2}
\left( \int_\cX \int_\cX \frac{ v^2(x)
p(x,z)}{v^2(z)} f^2(z) \, dz \, dx \right)^{1/2}. \]
Finally, using the fact that $v^2(x)$ is the stationary distribution shows that:
\[ \langle Af, Af \rangle \le \left( \int_\cX f(y)^2 \, dy \right)^{1/2}  \left(
\int_\cX f(z)^2 \, dz \right)^{1/2} = \|f\|^2, \]
as was to be shown.
\end{proof}

\section{Uniform boundedness of $\Delta_{g(\tau)}$ on
  $E_K$} \label{appendix: uniform bound}

We prove that the family of operators $\{ \Delta_{g(\tau)} \}_{0 \leq \tau
  \leq T}$ is uniformly bounded on $E_K \subset L^2 (\M)$, where
\begin{equation*}
E_K = \Span \{ \phi_l : 0 \leq l \leq K\}, 
\end{equation*}
and $\phi_l$ is the $l^{\text{th}}$ eigenfunction of $\Delta_{g(0)}$
with eigenvalue $\lambda_l$, ordered so that $0 = \lambda_0 <
\lambda_1 \leq \cdots \leq \lambda_K$. 

To simplify notation, set $\Delta_{\tau} = \Delta_{g(\tau)}: E_K \rightarrow
L^2(\M)$, and consider the function $\alpha (\tau) = \| \Delta_{\tau} \|$. If $\alpha
(\tau)$ is a continuous function in $\tau$, then $\alpha (\tau)$ is uniformly
bounded on $[0,T]$ since $\alpha (0) = \lambda_K <
\infty$ and $[0, T]$ is compact. It thus remains to show that $\alpha
(\tau)$ is a continuous function.

Let $\Delta_{\tau}^{\ast}: L^2(M) \rightarrow E_K$ be the adjoint of
$\Delta_{\tau}$. It suffices to show that $\beta (\tau) =
\| \Delta_{\tau}^{\ast} \Delta_{\tau} \|$ is a continous function in $\tau$, since
$\| \Delta_{\tau}^{\ast} \Delta_{\tau} \| = \| \Delta_{\tau} \|^2$. We have
$\Delta_{\tau}^{\ast} \Delta_{\tau}: E_K \rightarrow E_K$, and so the operator
$\Delta_{\tau}^{\ast} \Delta_{\tau}$ can be represented by the $(K+1) \times
(K+1)$ matrix $M_{\tau}$, defined through:
\begin{equation*}
\Delta_{\tau}^{\ast} \Delta_{\tau} \phi_j = \sum_{i=0}^{K+1} (M_{\tau})_{ij} \phi_i.
\end{equation*}
Recalling that the metric tensor $g(\tau)$ varies smoothly in $\tau$, it
follows that the entries of $M_{\tau}$ are continous in $\tau$ since
in local coordinates:
\begin{equation*}
\Delta_{\tau} f = \frac{1}{\sqrt{|g(\tau)|}} \partial_i \left( \sqrt{|g(\tau)|}
  g(\tau)^{ij} \partial_j f \right),
\end{equation*}
where $|g(\tau)|$ is the determinant of $g(\tau)$ in the local chart,
$g(\tau)^{ij}$ are the entries of the inverse of the metric tensor, and
the Einstein summation convention is used. But then the eigenvalues of
$M_{\tau}$ vary continuously in $\tau$, and in particular the operator norm of
$M_{\tau}$ is a continuous function of $\tau$. 

\end{document}